\documentclass[a4paper]{amsart}

\usepackage{verbatim}
\usepackage[dvips]{color}
\usepackage{amssymb}
\usepackage[colorlinks,pdfpagelabels,pdfstartview = FitH,bookmarksopen = true,bookmarksnumbered = true,linkcolor = blue,plainpages = false,hypertexnames = false,citecolor = red] {hyperref}
\usepackage[active]{srcltx}
\usepackage{todonotes}
\usepackage[final]{pdfpages} 

\allowdisplaybreaks

\newcommand{\R}{\mathbb{R}}
\title[Systems in porous medium]{Systems of partial differential equations\\ in porous medium}

\author[Kuusi]{Tuomo Kuusi}
\address{Tuomo Kuusi\\Aalto University
Institute of Mathematics
\\ P.O. Box 11100
FI-00076 Aalto,
Finland}
\email{tuomo.kuusi@aalto.fi}

\author[Monsaingeon]{L\'eonard Monsaingeon}
\address{L\'eonard Monsaingeon\\ CAMGSD,
Instituto Superior T\'ecnico, Universidade de Lisboa\\ Av. Rovisco Pais, 1049-001 Lisboa, Portugal}
\email{leonard.monsaingeon@ist.utl.pt}

\author[Videman]{Juha Videman}
\address{Juha Videman\\ CAMGSD/Departmento de Matem\'atica,
Instituto Superior T\'ecnico, Universidade de Lisboa\\Av. Rovisco Pais, 1049-001 Lisboa, Portugal}
\email{videman@math.ist.utl.pt}

\newtheorem{theorem}{Theorem}[section]
\newtheorem{prop}{Proposition}[section]

\newtheorem{lemma}{Lemma}[section]

\theoremstyle{definition}
\newtheorem{definition}{Definition}
\newtheorem{remark}{Remark}[section]

\numberwithin{equation}{section}

\def\eqn#1$$#2$${\begin{equation}\label#1#2\end{equation}}
\def\charfn_#1{{\raise1.2pt\hbox{$\chi_{\kern-1pt\lower3pt\hbox{{$\scriptstyle#1$}}}$}}}

\newcommand{\ro}{\varrho}

\newcommand{\eps}{\varepsilon}
\def\ff{\varphi}

 \DeclareMathOperator*{\osc}{osc}

\def\dist{\operatorname{dist}}
\def\dive{\operatorname{div}}

\newcommand{\divo}{\textnormal{div}}

\newcommand{\di}{\textnormal{dist}}

 \DeclareMathOperator*{\esssup}{ess\,sup}
  \DeclareMathOperator*{\essinf}{ess\,inf}

\def\en{\mathbb N}

\def\loc{\operatorname{loc}}

\def\mean#1{\mathchoice%
          {\mathop{\kern 0.2em\vrule width 0.6em height 0.69658ex depth -0.58065ex
                  \kern -0.8em \intop}\nolimits_{\kern -0.4em#1}}%
          {\mathop{\kern 0.1em\vrule width 0.5em height 0.69658ex depth -0.60387ex
                  \kern -0.6em \intop}\nolimits_{#1}}%
          {\mathop{\kern 0.1em\vrule width 0.5em height 0.69658ex
              depth -0.60387ex
                  \kern -0.6em \intop}\nolimits_{#1}}%
          {\mathop{\kern 0.1em\vrule width 0.5em height 0.69658ex depth -0.60387ex
                  \kern -0.6em \intop}\nolimits_{#1}}}

\def\vintslides_#1{\mathchoice%
          {\mathop{\kern 0.1em\vrule width 0.5em height 0.697ex depth -0.581ex
                  \kern -0.6em \intop}\nolimits_{\kern -0.4em#1}}%
          {\mathop{\kern 0.1em\vrule width 0.3em height 0.697ex depth -0.604ex
                  \kern -0.4em \intop}\nolimits_{#1}}%
          {\mathop{\kern 0.1em\vrule width 0.3em height 0.697ex depth -0.604ex
                  \kern -0.4em \intop}\nolimits_{#1}}%
          {\mathop{\kern 0.1em\vrule width 0.3em height 0.697ex depth -0.604ex
                  \kern -0.4em \intop}\nolimits_{#1}}}

\newcommand{\aveint}[2]{\mathchoice%
          {\mathop{\kern 0.2em\vrule width 0.6em height 0.69658ex depth -0.58065ex
                  \kern -0.8em \intop}\nolimits_{\kern -0.45em#1}^{#2}}%
          {\mathop{\kern 0.1em\vrule width 0.5em height 0.69658ex depth -0.60387ex
                  \kern -0.6em \intop}\nolimits_{#1}^{#2}}%
          {\mathop{\kern 0.1em\vrule width 0.5em height 0.69658ex depth -0.60387ex
                  \kern -0.6em \intop}\nolimits_{#1}^{#2}}%
          {\mathop{\kern 0.1em\vrule width 0.5em height 0.69658ex depth -0.60387ex
                  \kern -0.6em \intop}\nolimits_{#1}^{#2}}}
                  
\newtoks\by
\newtoks\paper
\newtoks\book
\newtoks\jour
\newtoks\yr
\newtoks\pages
\newtoks\vol
\newtoks\publ
\def\et{ \& }
\def\name[#1, #2]{#1 #2}
\def\ota{{\hbox{\bf ???}}}
\def\cLear{\by=\ota\paper=\ota\book=\ota\jour=\ota\yr=\ota
\pages=\ota\vol=\ota\publ=\ota}
\def\endpaper{\the\by, \textit{\the\paper},
{\the\jour} \textbf{\the\vol} (\the\yr), \the\pages.\cLear}
\def\endbook{\the\by, \textit{\the\book},
\the\publ, \the\yr.\cLear}
\def\endpap{\the\by, \textit{\the\paper}, \the\jour.\cLear}
\def\endproc{\the\by, \textit{\the\paper}, \the\book, \the\publ,
\the\yr, \the\pages.\cLear}

\begin{document}
\begin{abstract}
We investigate systems of degenerate parabolic equations idealizing reactive solute transport in porous media. Taking advantage of the inherent structure of the system that allows to  deduce a scalar Generalized Porous Medium Equation for the sum of the solute concentrations, we  show existence of a unique weak solution to the coupled  system and derive regularity estimates. We also prove that the system supports solutions  propagating with finite speed thus giving rise to free boundaries and interaction of compactly supported initial concentrations of different species. 

\end{abstract}
\maketitle

\section{Introduction}

The transport of pollutants in subsurface environments is a complex process  modeled by advection-diffusion-reaction equations that describe the evolution of contaminant concentrations in porous medium through advection, dispersion, diffusion and adsorption.  More than often the adsorption, accumulation of a pollutant on the  solid matrix at the fluid-solid interface, is in fact the main mechanism responsible for the contaminant transport  in soil. 

In this work, we address the case of multicomponent contaminant transport by considering a competitive adsorption process of Freundlich type between different species $z_1,\ldots z_N$. For expediency, we  write the Freundlich  multicomponent equilibrium isotherm as
$$
\mathbf{b}_a(\mathbf{z})=  |\mathbf{z}|_1^{p-1}\mathbf{z} \in \R^N\, ,  \quad p\in(0,1)\, , \quad \mathbf{z}=(z_1,\ldots,z_N)\, ,
$$
and consider the  model problem
\begin{equation}
\partial_t\mathbf{b}(\mathbf{z})=\Delta \mathbf{z}+\mathbf{f} ,
\label{eq:model}
\end{equation} 
where $|\mathbf{z}|_1=\sum\limits_{i=1}^N |z_i|$ is the usual $l^1$ norm in $\R^N$, $\mathbf{b}(\mathbf{z})=\phi\mathbf{z}+(1-\phi)\mathbf{b}_a(\mathbf{z})$, and  $\phi \in [0,1)$ is the medium constant porosity. 
 For a more thorough discussion on the physical background and derivation of this model, see Section~\ref{sec:phys} below. 
 
While  keeping in mind that the $\mathbf{b}$-term  in problem  \eqref{eq:model} arises from the multicomponent adsorption with the Freundlich isotherm, we will allow for more general  nonlinearities and denote hereafter  the Freundlich nonlinearity   as $\mathbf{b}_f(\mathbf{z})=(\phi+(1-\phi)|\mathbf{z}|_1^{p-1})\mathbf{z}$.

\subsection{General isotherms} \label{sec:ass}
Let us assume that
$$
\mathbf{b}(\mathbf{z})=  B(|\mathbf{z}|_1)\, \mathbf{z},
$$
where $B:\R^+\to\R^+$ is such that
$$
\beta(r):=B(|r|)r\in \mathcal{C}(\R)\cap\mathcal{C}^{1}(\R\setminus\{0\})
$$
and
\begin{equation}
\beta(0)=0,\quad\beta (\pm\infty)=\pm\infty,\quad\beta'(r)>0\text{ for }r\neq 0.
\label{eq:monotonicity_b}
\tag{$H_1$}
\end{equation}
Note that this  nonlinearity includes the Freundlich one $\mathbf{b}_f(\mathbf{z})=(\phi+(1-\phi)|\mathbf{z}|_1^{p-1})\mathbf{z}$, allows for blow up $D_{\mathbf z}\mathbf{b}(0)\sim\infty$, and that by definition $|\mathbf{b}(\mathbf{z})|_1=\beta(|\mathbf{z}|_1)$.  Most importantly, this type of nonlinearity possesses a structure which will allow us to derive a particular scalar equation from the system. Since $\beta$ is monotone increasing, the inverse
$$
\Phi=\beta^{-1}
$$
is well-defined and continuous with $\Phi(0)=0$. We further assume that $\Phi\in\mathcal{C}^1(\R)$ satisfies the structural conditions
\begin{equation}
s\in \R:\qquad 1 \leq \frac{s\Phi'(s)}{\Phi(s)} \leq \frac{1}{a}
\label{eq:structural_diffusion_hyp}
\tag{$H_2$} 
\end{equation}
and
\begin{equation}
s\in \R:\qquad \frac{s\Phi''(s)}{\Phi'(s)} \geq  - \frac{1}{a}
\label{eq:convexity_hyp}
\tag{$H_3$}
\end{equation}
for some structural constant $a\in (0,1)$. In the scalar case $\partial_t b(z)=\Delta z+(\ldots)$ the change of variables $u=b(z)$ produces the well-known Generalized Porous Media Equation (GPME) $\partial_tu=\Delta \Phi(u)+(\ldots)$ for which \eqref{eq:structural_diffusion_hyp} is a standard assumption, see~\cite{DK07,Va07} and the references therein. It is well-known that the scalar GPME is degenerate at $u=0$ if $\Phi'(0)=0$ and strictly parabolic if $\Phi'(0)>0$, which can be seen from the divergence form $\Delta\Phi(u)=\dive(\Phi'(u)\nabla u)$. Note in particular that the structural lower bound in \eqref{eq:structural_diffusion_hyp} includes both the degenerate \emph{slow diffusion} $\Phi'(0)=0$ and the nondegenerate case $\Phi'(0)>0$, but the assumption $\Phi\in\mathcal{C}^1(\R)$ excludes \emph{fast diffusion} $\Phi'(0)=+\infty$. We will shortly perform a similar change of variables $\mathbf{u}=\mathbf{b}(\mathbf{z})$ for the multicomponent problem in order to move the nonlinearity from the time derivative to the spatial ones.

We shall deal with the degenerate and the nondegenerate diffusions simultaneously in a unified framework, except in Section \ref{section:FB} where we discuss the existence of free boundaries and consequently restrict ourselves to \emph{slow diffusions} $\Phi'(0)=0$. Note also that in view of \eqref{eq:structural_diffusion_hyp} the function $\Phi(s)/s$ is continuous monotone and nondecreasing in $\R$ with at most algebraic growth
\begin{equation}
0<s_1<s_2:\qquad \frac{\Phi(s_2)/s_2}{\Phi(s_1)/s_1}\leq \left(\frac{s_2}{s_1}\right)^{\frac{1}{a}-1}\,,
\label{eq:algebraic_growth_Phi(s)/s}
\end{equation}
all properties that  will be crucial in the subsequent analysis just as in the standard theory for GPME. The structural assumptions \eqref{eq:monotonicity_b}-\eqref{eq:structural_diffusion_hyp}-\eqref{eq:convexity_hyp} are easily verified for the physical Freundlich isotherm when $p\in (0,1)$. With these structural assumptions the blowup $D_{\mathbf z}\mathbf{b}(0)\sim\infty$ corresponds now to slow diffusion $\Phi'(0)=0$, but linear diffusion $\Phi(s)=s$ is also allowed. In fact, the Freundlich isotherm $\mathbf{b}_f(\mathbf{z})=(\phi +(1-\phi)|\mathbf{z}|_1^{p-1})\mathbf{z}$ behaves like $\phi\mathbf{z}$ for large $|\mathbf{z}|_1$, hence $\beta(r)\sim \phi r$ and $\phi(s)=\beta^{-1}(s)\sim s/\phi$ for large $r,s$. In other words, the system $\partial_t\mathbf{b}_f(\mathbf{z})=\Delta \mathbf{z}+\mathbf{f}$ behaves as $N$ uncoupled linear heat equations $\phi \partial_t \mathbf{z} \approx \Delta \mathbf{z}+\mathbf{f}$ for large $|\mathbf{z}|_1$. From the physical point of view, roughly speaking, this means that for very large concentrations the porous rock matrix saturates and 
the adsorption phenomena become negligible compared to inertial effects.

The monotonicity assumption \eqref{eq:monotonicity_b} allows us to invert 
$$
\mathbf{b}(\mathbf{z})=\mathbf{u}\,\Leftrightarrow \, \mathbf{z}=\frac{\Phi(|\mathbf{u}|_1)}{|\mathbf{u}|_1}\mathbf{u},
$$
so that  problem (\ref{eq:model}) can be recast as a degenerate parabolic system of (generalized) porous medium type
\begin{equation}
\partial_t\mathbf{u}  = \Delta \, \left( \frac{\Phi(|\mathbf{u}|_1)}{|\mathbf{u}|_1}\mathbf{u} \right)+\mathbf{f}. 
\label{eq:model_PME}
\end{equation}
We shall refer to $\mathbf{u}$ as density whereas we shall speak of concentration when dealing with the original $\mathbf{z}$ variable.\\

\subsection{Systems of decoupled Cauchy-Dirichlet problems}
Let $\Omega\subset \R^n$  be a bounded open set with smooth boundary $\partial\Omega$ and define  $Q_T=\Omega\times(0,T)$ and $\Sigma_T=\partial\Omega\times(0,T)$ for fixed $T>0$. For given boundary data $\mathbf{z}^D(x,t)=(z^D_1,\ldots,z^{D}_N)(x,t)$, initial condition $\mathbf{z}^0(x)=(z^0_1,\ldots,z^0_N)(x)$, and the resultant of forcing terms $\mathbf{f}(x,t)=(f_1,\ldots,f_N)(x,t)$, we consider the following two equivalent formulations, the first written for the original concentrations $\mathbf{z}=(z_1,\ldots,z_N)$
\begin{equation}
% \begin{cases}
%  \partial_t \mathbf{b}(\mathbf{z})=\Delta \mathbf{z}+\mathbf{f} \quad \mbox{in } \, Q_T,\\
%  \mathbf{z}\big|_{\Omega \times \{0\}} =  \mathbf{z}^0, \\
% \mathbf{z}\big|_{\partial \Omega \times (0,T)} =  \mathbf{z}^D\,.
% \end{cases}
\left\{
\begin{array}{ll}
 \partial_t \mathbf{b}(\mathbf{z})=\Delta \mathbf{z}+\mathbf{f}	 & \mbox{in } \, Q_T\\
 \mathbf{z}(x,0) =  \mathbf{z}^0 (x)	& \mbox{in }\Omega \\
\mathbf{z} =  \mathbf{z}^D	& \mbox{in }\Sigma_T
\end{array}
\right.
\label{eq:PB_isotherm_z} 
\end{equation}
and the second one for the  densities $\mathbf{u}=(u_1,\ldots,u_N)$
\begin{equation}
% \begin{cases}
%  \partial_t \mathbf{u}=\Delta\left(\frac{\Phi(|\mathbf{u}|_1)}{|\mathbf{u}|_1}\mathbf{u}\right)+\mathbf{f} \quad \mbox{in } \, Q_T,\\
% \mathbf{u}\big|_{\Omega \times \{0\}} =  \mathbf{u}^0, \\
% \frac{\Phi(|\mathbf{u}|_1)}{|\mathbf{u}|_1}\mathbf{u}\big|_{\partial \Omega \times (0,T)} =  z^D\,
% \end{cases}
%
\left\{
\begin{array}{ll}
 \partial_t \mathbf{u}=\Delta\left(\frac{\Phi(|\mathbf{u}|_1)}{|\mathbf{u}|_1}\mathbf{u}\right)+\mathbf{f}	 & \mbox{in } \, Q_T\\
 \mathbf{u}(x,0) =  \mathbf{u}^0 (x)	& \mbox{in }\Omega \\
\frac{\Phi(|\mathbf{u}|_1)}{|\mathbf{u}|_1}\mathbf{u} =  \mathbf{z}^D	& \mbox{in }\Sigma_T
\end{array}
\right.
\label{eq:PB_u} 
\end{equation}
where $\mathbf{b}(\mathbf{z})=\mathbf{u}\Leftrightarrow \mathbf{z}=\frac{\Phi(|\mathbf{u}_1|)}{|\mathbf{u}|_1}\mathbf{u}$. As usual for the scalar GPME,  the boundary conditions in the density formulation \eqref{eq:PB_u} are enforced in terms of the physical concentration $\frac{\Phi(|\mathbf{u}|_1)}{|\mathbf{u}|_1}\mathbf{u}=\mathbf{b}^{-1}(\mathbf{u})=\mathbf{z}^D$ rather than $\mathbf{u}=\mathbf{u}^D=\mathbf{b}(\mathbf{z}^D)$. It is easy to see formally  that non-negative data $f_i,u^0_i,z^D_i\geq 0$ should lead to non-negative solutions $u_i\geq 0\iff z_i\geq 0$ and, therefore, we shall only deal with such non-negative data and solutions. This is of course consistent with the fact that $z_i$ represent physical concentrations and should stay non-negative when time evolves. Summing the equations  in \eqref{eq:PB_u} we recognize that $w=|\mathbf{u}|_1=u_1+\ldots+u_N$ is a non-negative solution to
\begin{equation}
w=|\mathbf{u}|_1:\qquad
\left\{
\begin{array}{ll}
\partial_t w=\Delta(\Phi(w))+F & \mbox{in } \, Q_T,\\
 w(x,0) =  w^0(x) 	& \mbox{in }\Omega, \\
\Phi(w)=  g^D	& \mbox{in }\Sigma_T\,
\end{array}
\right.
\label{eq:PB_w}
\tag{GPME}
\end{equation}
with $F=f_1+\ldots + f_N\geq 0$, $g^D=|\mathbf{z}^D|_1\geq 0$, and $w^0=|\mathbf{u}^0|_1\geq 0$.  Note that the boundary data is written for $\Phi(w)$ rather than for $w$ as is common for the scalar GPME.\\

The initial condition and inhomogeneity should satisfy
\begin{equation}
\forall i=1\ldots N:\ \  0 \leq u^0_i\leq M \text{ and }0\leq f_i\leq M \quad \text{a.e. }(x,t)\in Q_T
\label{hyp:initial+forcing}
\end{equation}
for some finite $M>0$. The boundary data will always assumed to be non-negative and bounded as well, but we shall sometimes assume the following. If $\gamma:\Omega\to \partial\Omega$ is the usual trace operator then there exists $\mathbf{Z}^D(x,t)$ such that
\begin{equation}
\label{hyp:boundary_data}
\mathbf{z}^D=\gamma(\mathbf{Z}^D):\ \  0\leq Z_i^D\in L^{\infty}(Q_T)\cap L^{2}(0,T;H^1(\Omega))\text{ and }\partial_t Z_i^D\in L^{\infty}(Q_T)
\end{equation}
(we shall indistinctly write $\mathbf{z}^D$ or $\mathbf{Z}^D$ both for the trace boundary values or their extension to $\Omega$). 
% \begin{remark}
% It is worth stressing  that we do not impose any compatibility condition on, or assume regularity of, the boundary and initial data at this stage. As a matter of fact  we shall later  establish  interior regularity estimates for the solution % (Proposition~\ref{prop:C_alpha_estimate}) 
% for merely bounded data, and extend the interior regularity up to the boundaries in the case of regular and compatible data. % (Proposition~\ref{prop:boundary_regularity}).
% \end{remark}

\subsection{Main results}

Let us now first introduce our main theorem, which addresses existence, uniqueness, and regularity.  
\begin{theorem}
Assume that \eqref{eq:structural_diffusion_hyp} holds. For any data $0\leq u_i^0\leq M$, $0\leq z^D_i\leq M$, and $0\leq f_i\leq M$ there exists a unique non-negative bounded very weak solution $\mathbf{u}$ to \eqref{eq:PB_u}. Moreover, $w=|\mathbf{u}|_1$ is the unique non-negative bounded very weak solution to \eqref{eq:PB_w} and there exist positive constants $\alpha=\alpha(a,n)\in (0,1)$  and  $C=C(a,T,n,N,M)$ such that
\begin{equation}
\|\mathbf{u}\|_{C^{\alpha,\alpha/2}(Q')}\leq C(1+1/d'+1/\sqrt{\tau})
\label{eq:uniform_Holder}
\end{equation}
holds in all parabolic subdomains  $Q'=\Omega'\times(\tau,T)$ with $0<\tau<T$, $\Omega'\subset\subset \Omega$, and $d'=\di(\overline{\Omega'},\partial\Omega)$.

Assume in addition that \eqref{eq:convexity_hyp} holds and that the data satisfy \eqref{hyp:initial+forcing}-\eqref{hyp:boundary_data}. Then $w$ is a global weak energy solution to \eqref{eq:Cauchy_PB_w} and $\mathbf{u}$ is a local weak energy solution to \eqref{eq:Cauchy_PB_u} in the sense that
\begin{equation}
 \|\nabla(\ro u_i)\|_{L^2(Q'_T)} \leq C(1+1/d')\qquad  \forall\, i=1\ldots N\, , 
\label{eq:energy_very_weak_ui}
\end{equation}
where $\ro=\frac{\Phi(w)}{w}=\frac{\Phi(|\mathbf{u}|_1)}{|\mathbf{u}|_1}$, holds in any $Q'_T=\Omega'\times(0,T)$ with some constant $C=C(a,T,n,N,M)>0$. 
\label{theo:exist_sols_Dirichlet_Cauchy_PB}
\end{theorem}
\noindent 
The proof of the theorem can be found from Section~\ref{section:weak_sols}, where one also find definitions of very weak and energy solutions. It is also worth stressing that if the initial and boundary data are compatible, then the local regularity \eqref{eq:uniform_Holder} can be improved to global regularity up to the bottom and lateral boundaries, see Proposition~\ref{prop:boundary_regularity} later on. 

Let us immediately comment the content of the theorem. First of all, in the theory of (possibly degenerate) \emph{scalar}  diffusion equations such as \eqref{eq:PB_w} the so-called \emph{pressure} variable $p=p(w)=\Phi'(w)$ plays an important role, as can be seen from the divergence form $\Delta\Phi(w)=\dive(\Phi'(w)\nabla w)$. In  view of \eqref{eq:PB_u}  another pressure variable of interest is clearly
\begin{equation}
\ro=\ro(w)=\frac{\Phi(w)}{w},\qquad w=|\mathbf{u}|_1.
\label{eq:def_pressure_rho}
\end{equation}
Note that our structural assumption \eqref{eq:structural_diffusion_hyp} bounds the ratio $p/\ro$ away from zero and from above, roughly meaning that the degeneracy   of $p$ in \eqref{eq:PB_w} should be comparable to the degeneracy of $\ro$ in \eqref{eq:PB_u}. The natural energy space for \eqref{eq:PB_w} is  $\nabla \Phi(w)=p\nabla w=\nabla(\ro w)\in L^2(Q_T)$, whereas that for \eqref{eq:PB_u} is rather here $\nabla  (\ro u_i) \in L^2(Q_T)$. Our structural assumption \eqref{eq:convexity_hyp} will later allow us to derive such an estimate for each $u_i$ from that estimate obtained for $w$. In other words, bounds for the scalar quantity $w$ suffice to control the system in terms of energy considerations. This idea of controlling the vector-valued $\mathbf{u}$ by means of the scalar $w$ will be the cornerstone of our analysis, and will also appear in the study of the H\"older regularity and of the free-boundaries.

Second, concerning the existence, in their celebrated work \cite{AL}, Alt and Luckhaus studied systems of elliptic-parabolic PDEs  which included the $\mathbf{b}$-term as in \eqref{eq:model}  and also nonlinear $p$-Laplacian type diffusion, Stefan problems, and reaction terms $\mathbf{f}=\mathbf{f}(x,t,\mathbf{z})$. Their analysis requires however a particular monotone structure which restricts the  $\mathbf{b}$-term  to be of the form $\mathbf{b}=D_{\mathbf{z}}\varphi$ for some convex potential $\varphi$ satisfying certain structural assumptions. In our case, the dependence of the Freundlich nonlinearity $\mathbf{b}_f$ on $\mathbf{z}$ through  the $l^1(\R^N)$-norm precludes any such monotonicity and, therefore, the results of \cite{AL} seem to be of no use here. Though system \eqref{eq:Cauchy_PB_u} is formally parabolic for non-negative solutions, it is readily checked that the ellipticity fails for signed solutions due to the dependence on the $l^1(\R^N)$ norm.
A direct approach by Galerkin approximation as in \cite{AL} produces here approximative solutions whose componentwise sign cannot be controlled uniformly. Since ellipticity fails for signed solutions the sequence of projected solutions does not enjoy enough compactness, hence the method from \cite{AL} cannot be adapted. In order to tackle this issue we use instead the specific structure of the system allowing to control each component $u_i$ in terms of the scalar quantity $w=|\mathbf{u}|_1$.

The method of proof for Theorem~\ref{theo:exist_sols_Dirichlet_Cauchy_PB} is classical for scalar problems, but requires technical work for system \eqref{eq:Cauchy_PB_u}: we first establish existence of positive classical solutions $\mathbf{u}^k$ for approximated positive data  $u^{0.k},\mathbf{z}^{D,k},\mathbf{f}^k$ and derive some a priori regularity and energy estimates. Taking $k\to\infty$ finally gives the desired solution $\mathbf{u}=\lim \mathbf{u}^k$, which inherits regularity and energy estimates from the previous ones.\\

\subsection{Physical background} \label{sec:phys}

Based on  a continuum approach at a macroscopic level by homogenization, the mass conservation written for the concentration  of one contaminant component $z=z(t,x)$  can be written as, cf. \cite{BV87},
\begin{equation}
\phi\frac{\partial  z}{\partial t} +\rho (1-\phi)\,\frac{\partial  b_a}{\partial t}  +  \phi\, \nabla\cdot (z\mathbf{V} -\, \mathbf{D}  \nabla z ) = f\, .
\label{eq:real_model}
 \end{equation}
Here, we have made the assumptions of  saturated  flow and constant porosity $\phi\in(0,1)$,
denoted the advective water flux by $z\mathbf{V}$  ($\mathbf{V}$ is the  Darcy velocity),  the  bulk density of the solid matrix  by $\rho>0$, the hydrodynamic dispersion matrix describing both the molecular diffusion and the mechanical dispersion by $\mathbf{D}$,  and  modeled the source or sink terms by $f$. Moreover, $b_a=b_a(z)$ describes the concentration of contaminant adsorbed on the solid matrix through  a reactive adsorption process that can be assumed to be either fast  (equilibrium) or slow (non-equilibrium).  

An adsorption isotherm $b_a(z)$ relates the concentration of the adsorbed component to its concentration in the fluid phase at constant temperature. One of the most commonly used nonlinear equilibrium isotherms for a single species is the Freundlich  isotherm expressed as, cf. \cite{BV87,WMK91}
$$
b_a(z)=K\, z^{p}\, \quad K>0\, , \quad p\in(0,1)\, , \quad z\geq 0\, .
$$

The Freundlich exponent $p\in(0,1)$ makes equation \eqref{eq:real_model} singular at $z=0$ because, at least formally, $\partial_t b_a(z)=b_a^\prime(z)\partial_t z$ and $b_a^\prime(0)=\infty$. The equation may thus exhibit finite speed of propagation  of compactly supported initial solutions giving rise to free boundaries that separate the region where the solute concentration vanishes from that with positive concentration. This  is in marked contrast with  the   behavior  of  solutions when  the Freundlich exponent  $p$ equals or exceeds one since  the equation becomes nonsingular for $p\geq 1$ and the information propagates with infinite speed, as usual for uniformly parabolic equations.

Equation (\ref{eq:real_model}), complemented with suitable initial and boundary conditions,  and all its variants arising from different equilibrium and non-equilibrium, linear or non-linear, isotherms, have attracted  considerable attention  over the last 20 years, both from an analysis and numerical simulation point of view, see, \emph{e.g.},  \cite{vDK91,DvDW94, DvDG96,BK97a,BK97b,KO2000,ADCC01}.  It is, however,  the above equilibrium Freundlich isotherm which  makes the problem most challenging due to the degeneracy and similarity to the porous medium equation. 

A competitive adsorption process  between different species $z_1,\ldots z_N$,  can be modeled by a multicomponent isotherm  of Freundlich type, cf. \cite{SRS81,GF93, WKLL02}, and  \cite{SMM06} for a review of competitive equilibrium adsorption modeling. As an idealization of the physical model proposed, e.g.,  in \cite{SRS81}, the multicomponent  Freundlich  equilibrium isotherm can be expressed as
$$
\mathbf{b}_a(\mathbf{z})=  |\mathbf{z}|_1^{p-1}\mathbf{z} \in \R^N\, ,  \quad p\in(0,1)\, , \quad \mathbf{z}=(z_1,\ldots,z_N)\, ,
$$
where $|\mathbf{z}|_1=\sum\limits_{i=1}^N |z_i|$. 
Scaling the density, neglecting the hydrodynamical effects  ($\mathbf{V}=0)$, assuming that $\mathbf{D}=\operatorname{Id}$  and writing $\mathbf{b}(\mathbf{z})=\phi\mathbf{z}+(1-\phi)\mathbf{b}_a(\mathbf{z})$,  we obtain from  (the multicomponent version of)  \eqref{eq:real_model} the  model problem~\eqref{eq:model}, i.e., 
$$
% \begin{equation}
\partial_t\mathbf{b}(\mathbf{z})=\Delta \mathbf{z}+\mathbf{f} \, .
% \label{eq:model 2}
% \end{equation}
$$
% In what follows,  $\mathbf{b}(\mathbf{z})$ may stand for a more general nonlinearity than the Freundlich one $\mathbf{b}_f(\mathbf{z})=(\phi+(1-\phi)|\mathbf{z}|_1^{p-1})\mathbf{z}$.
%

\subsection{The content}
The paper is organized as follows. In Section \ref{section:smooth_sols} we consider smooth positive data, construct corresponding smooth positive solutions to \eqref{eq:PB_u}, and establish a priori energy as well as H\"older estimates. The H\"older estimates are based on the celebrated method of intrinsic scaling \cite{DBb}, a standard technique at least for scalar problems. In Section~\ref{section:weak_sols} we consider more general data, introduce  different notions of weak solutions, and prove Theorem~\ref{theo:exist_sols_Dirichlet_Cauchy_PB}. Approximating the data suitably we show existence of a unique weak solution to problem \eqref{eq:PB_u} which inherits H\"older regularity and energy estimates from the  smooth positive solutions constructed in Section~\ref{section:smooth_sols}. Finally in Section~\ref{section:FB} we impose the degeneracy condition $\Phi'(0)=0$ and consider the problem in the whole space without the forcing term and with compactly supported initial data. We show that the 
corresponding  Cauchy problem is well posed and admits free boundary solutions and, moreover, investigate the finite speed of 
propagation of the  free boundaries and the evolution and interaction of distinct  compactly supported initial concentrations.
%
%
%
%%%%%%%%%%%%%%%%%%%%%%%%%%%%%%%%%%%%%%%%%%%%%%%%%%%%%%%%%%%%%%%%%%%%%%%%%%%%%%%%%%%%%%%%%%%%%%
%
%
\section{smooth positive solutions and a priori estimates}
\label{section:smooth_sols}
We will assume throughout this section that the data is smooth and (componentwise) positive.  Solutions of \eqref{eq:PB_u}  and \eqref{eq:PB_w} corresponding to such data are shown to be classical and positive and satisfy certain a priori energy and locally uniform H\"older estimates.
\begin{prop}[Existence of positive classical solutions]
Assume that $\mathbf{z}^D$ and $\mathbf{u}^0$ are smooth and positive and $\mathbf{f}$ is smooth and non-negative. Moreover,  let $F:=|\mathbf{f}|_1$ and assume that
\begin{align*}
0 < m  = \min \left\{ \essinf \limits_{\overline{Q_T}}|\mathbf{u}^0|_1, \quad \essinf \limits_{\overline{\Sigma_T}}|\mathbf{z}^D|_1\right\},\\
0 < M  = \max \left\{ \esssup \limits_{\overline{Q_T}}|\mathbf{u}^0|_1, \quad \esssup \limits_{\overline{\Sigma_T}}|\mathbf{z}^D|_1, \quad \esssup \limits_{\overline{Q_T}}F\right\}.
\end{align*}
Then there exists a classical solution $\mathbf{u}\in \mathcal{C}^{2,1}(\overline{\Omega}\times(0,T))\cap\mathcal{C}^{2,1}(\Omega\times[0,T])\cap \mathcal{C}^{\infty}(Q_T)$ to \eqref{eq:PB_u} with $u_i>0$ on $\overline{Q_T}$. Moreover, defining $w=|\mathbf{u}|_1=u_1+\ldots+u_N$,  $w \in \mathcal{C}^{2,1}(\overline{Q_T})\cap \mathcal{C}^{\infty}(Q_T)$ is a classical solution to \eqref{eq:PB_w} and
$$
0 < m  \leq w(x,t) \leq M (1+T)\quad \text{in }\overline{Q_T}.
$$
\label{prop:exists_classical_solutions}
\end{prop}
\begin{remark} We do not impose any compatibility conditions on the initial and boundary data at $\partial\Omega\times\{t=0\}$. Although this limits the boundary regularity  it has no importance in the sequel. Note also that we could  prove uniqueness of positive classical solutions at this stage. However since we will later establish a stronger uniqueness result (within the class of non-negative very weak solutions)  we postpone the uniqueness issue until then.
\end{remark}
\begin{proof}
We will exploit the diagonal structure of the system by first showing  existence of a classical solution $w$ to \eqref{eq:PB_w}, then reconstructing $\mathbf{u}$ by solving $N$ independent linear parabolic equations for the $u_i$, and finally checking  that $w=|\mathbf{u}|_1$ as desired.
\par
For smooth positive data let $w^0:=|\mathbf{u}^0|_1>0$ in $\overline{\Omega}$ and $g^D:=|\mathbf{z}^D|_1>0$ in $\overline{\Sigma_T}$. Write $\Delta\Phi(w)=\dive (\Phi'(w)\nabla w)$ and observe that hypothesis \eqref{eq:structural_diffusion_hyp} implies that $\Phi'(w)>0$ is bounded away from zero and from above as long as $0<m \leq w\leq C$ so that  equation \eqref{eq:PB_w} is uniformly parabolic for such values of $w$. Therefore, after approximating $\Phi$ by a globally Lipschitz function $\Phi_{\eps}$ such that $\Phi_{\eps}(0)=0$, $\Phi(s)=\Phi_{\eps}(s)$ for $|s|\in(\eps,1/\eps)$ and $\Phi_{\eps}'>0$, well known results for quasilinear parabolic equations (cf. \cite{LSU}) guarantee the existence of a positive classical solution $w_{\eps}(x,t)$ to the $\eps$-problem. A standard comparison principle with $0\leq F\leq M $ and $0<m \leq w^0,g^D\leq M $ shows moreover that   
\[
0<m \leq w_{\eps}\leq \max\left\{\|w^0\|_{L^{\infty}(\Omega)},\|g^D\|_{L^{\infty}(\Sigma_T)}\right\}+T\|F\|_{L^{\infty}(Q^T)}\leq M (1+T)
\]
 as in our statement. In particular, for $\eps>0$ small 
enough there holds $\Phi_{\eps}(w_{\eps})=\Phi(w_{\eps})$ so that $w_{\eps}$ 
is  in  fact   a classical solution to the original  problem. The argument is standard for scalar equations and we refer, e.g., to \cite{Va07,DK07} for more details.
\par
Once there exists  a smooth positive solution $w$ to \eqref{eq:PB_w}, the pressure $\ro = \frac{\Phi(w)}{w}$ becomes smooth in the interior, belongs to $\mathcal{C}^{2,1}(\overline{\Omega}\times(0,T))\cap\mathcal{C}^{2,1}(\Omega\times[0,T])$, and the bounds
\begin{equation}
0<C_1\leq \ro \leq C_2\quad \mbox{in }\overline{Q^T}
\label{eq:rho_ui}
\end{equation}
hold for some $C_1,C_2>0$ depending on $a,m ,M ,T$ only. Then standard results \cite{LSU} on \emph{linear} parabolic equations allow us to solve
\begin{equation}
\left\{
\begin{array}{ll}
\partial_t u_i = \Delta(\ro u_i)+f_i = \dive(\ro \nabla u_i )+ \dive(u_i \nabla \ro)+f_i \qquad & \text{in }Q_T,\\
 \ro u_i = z^D_i	& 	\text{in }\Sigma_T,\\
 u_i(x,0) =u^0_i(x)		&	\text{in }\Omega\, ,
\end{array}
\right.
\label{eq:rho_bnd}
\end{equation}
for fixed $i=1,\ldots N$ and show that $u_i\in \mathcal{C}^{2,1}(\overline{\Omega}\times(0,T))\cap\mathcal{C}^{2,1}(\Omega\times[0,T])\cap \mathcal{C}^{\infty}(Q_T)$ (up to the corners if the  data are compatible). Indeed, from \eqref{eq:rho_bnd} it follows that the equation is uniformly parabolic and the  boundary condition reads simply as $u_i=\frac{z^D_i}{\ro}$ on $\Sigma_T$. The assumptions on the data and the strong maximum principle  ensure moreover that $u_i>0$ in $\overline{Q_T}$.
\par
Let now $\tilde{w}=|\mathbf{u}|_1$ and observe that  $\partial_tw=\Delta\Phi(w)+F=\Delta(\ro w)+F$. Because $u_i>0$ we can write $\tilde{w}=u_1+\ldots +u_N$. Summing \eqref{eq:rho_ui} over $i=1\ldots N$, we obtain $\partial_t \tilde{w}=\Delta(\varrho \tilde{w})+F$. In other words,   $\tilde{w}$ is a positive classical solution to the same equation  as $w$ with the same initial and boundary data. By standard uniqueness argument for smooth positive solutions we conclude that $w=\tilde{w}$. In particular $\ro = \frac{\Phi(w)}{w} = \frac{\Phi(\tilde{w})}{\tilde{w}}=\frac{\Phi(|\mathbf{u}|_1)}{|\mathbf{u}|_1}$ in \eqref{eq:rho_ui} and the proof is complete.
\end{proof}
\begin{prop}[A priori energy estimates]
Assume that  hypotheses \eqref{eq:monotonicity_b}, \eqref{eq:structural_diffusion_hyp} and \eqref{eq:convexity_hyp} hold, let $\mathbf{u}\in \mathcal{C}^{2,1}(\overline{Q_T})$ be a classical positive solution corresponding to smooth positive data and assume that
$$
\|\mathbf{u}^0\|_{L^{\infty}(\Omega)} + \|\mathbf{z}^D\|_{L^{\infty}(0,T;H^1(\Omega))}+\|\partial_t \mathbf{z}^D\|_{L^{\infty}(Q_T)}+\|\mathbf{f}\|_{L^{\infty}(Q_T)}\leq M,
$$
for some $M>0$.  Then we have the bounds
\begin{equation}
\| \nabla (\ro w) \|_{L^2(Q_T)}\leq C,
\label{eq:energy_w}
\end{equation}
where $w=|\mathbf{u}|_1$ and $\ro=\frac{\Phi(w)}{w}$, and
\begin{equation}
\qquad \|\nabla (\varrho  u_i)\|_{L^2(Q'_T)}\leq C(1+1/d'), \qquad \forall\,i=1\ldots N \, .
\label{eq:energy_rho_ui}
\end{equation}
\label{prop:energy_estimate}
Here, the constant $C>0$ depends on $a,T,n,N,M$ only, and $Q^\prime_T=\Omega^\prime\times (0,T)$ with
 $\Omega'\subset\subset\Omega$ and $d'=\operatorname{dist}(\overline{\Omega'},\partial\Omega)>0$. 
\end{prop}
\begin{remark}
We were not able to establish  energy estimates for $\nabla(\ro u_i)$ up to the boundary  as the  dependence on $1/d^\prime$ in  \eqref{eq:energy_rho_ui} shows. This will not be an issue  later on since   
in Proposition~\ref{prop:uniqueness_weak_sols} we shall prove uniqueness within the class of very weak solutions  and estimate \eqref{eq:energy_rho_ui} as well as 
 assumption  \eqref{eq:convexity_hyp},
which is only used in proving \eqref{eq:energy_rho_ui},  can be dispensed with while considering very weak solutions. On the other hand,  estimate \eqref{eq:energy_rho_ui} is sufficient for our purposes in Section~\ref{section:FB} where   the problem  is considered in $\R^n$ with compactly supported  initial data.
\label{rmk:cst=1/d'}

Observe also that   the validity of  estimate \eqref{eq:energy_rho_ui} up to the boundary would directly yield  \eqref{eq:energy_w} since $w=|\mathbf{u}|_1=\sum_i u_i$  for non-negative solutions.

\end{remark}
\begin{proof}
We will first establish  \eqref{eq:energy_w} for the scalar variable $w$ and then show how the particular structure of  system  \eqref{eq:PB_u}  allows us to derive  \eqref{eq:energy_rho_ui} for each component $u_i$.  We shall denote by $C$ any positive constant depending, as in the statement, only on $a,T,n,N,M$ whereas the primed constants $C'$ are also allowed to depend on $d'=\text{dist}(\overline{\Omega'},\partial\Omega)$. \\

\noindent
{\bf Step 1.} The assumptions on  $\mathbf{u}^0,\mathbf{z}^D,\mathbf{f}$ translate into similar properties for the data $w^0=|\mathbf{u}^0|_1,g^D=|\mathbf{z}^D|_1,F=|\mathbf{f}|_1$  so by the comparison principle for  solutions to \eqref{eq:PB_w} we have
$$
0\leq u_i\leq |\mathbf{u}|_1=w\leq C(M,N,T).
$$
Recalling that  $\ro w=\Phi(w)$,  inequality \eqref{eq:energy_w} is nothing but the usual (global) energy estimate for the GPME leading to the usual concept of weak \emph{energy} solutions. For smooth positive solutions, bound \eqref{eq:energy_w} is easily derived for $\|\nabla \Phi(w)\|_{L^2(Q_T)}=\|\nabla (\ro w)\|_{L^2(Q_T)}$ by  taking $\ff=(\Phi(w)-g^D)\in L^2(0,T;H^1_0(\Omega))$ as a test function in \eqref{eq:PB_w}, cf.  \cite{DK07,Va07} for further details.\\

\noindent
{\bf Step 2.}  Since $0\leq w\leq C$, the structural assumptions imply that
$$
0\leq \ro =\frac{\Phi(w)}{w}\leq \frac{1}{a}\Phi'(w)\leq C(a,M,T).
$$
The $L^{\infty}(Q_T)$-norm of any term involving $u_i,w,\ro$ can thus be bounded  by a constant $C=C(a,T,n,N,M)>0$ only. Now fix $i\in\{1,,\ldots,N\}$ and choose  a cutoff function $\chi=\chi(x)\in \mathcal{C}^{\infty}_c(\Omega)$ such that $0\leq \chi\leq 1$ in $\Omega$, $\chi\equiv 1$ in $\Omega'$ and $|\nabla \chi|\leq 2/d'$ where $\Omega'\subset\subset\Omega$ and $d'=\text{dist}(\overline{\Omega'},\partial\Omega)$. Multiplying the $i$-th equation in \eqref{eq:PB_u} by  a test function $\ff=\chi^2\ro u_i$, integrating over $Q_T$ and  by parts in the Laplacian term, we obtain
\begin{align*}
\int\limits_{Q_T}\chi^2|\nabla (\ro u_i)|^2\,\mathrm{d}x \, \mathrm{d}t	 & = - 2\int\limits_{Q_T}\ro u_i \chi \nabla \chi \cdot \nabla (\ro u_i) \,\mathrm{d}x \, \mathrm{d}t  \\ & +  \int\limits_{Q_T}\chi^2 \ro u_i f_i	 \,\mathrm{d}x \, \mathrm{d}t
- \int\limits_{Q_T}\chi^2\ro u_i\partial_t u_i	\,\mathrm{d}x \, \mathrm{d}t.
\end{align*}
Integrating the last term  by parts in $t$  and using $0\leq u_i,\ro\leq C$ to bound the  limit terms at $t=0,T$, gives
\begin{align*}
 \int\limits_{Q_T}\chi^2 & |\nabla (\ro u_i)|^2 \,\mathrm{d}x \, \mathrm{d}t	 \\
 &	 \leq 2 \|\ro u_i\nabla \chi\|_{L^2(Q_T)} \|\chi \nabla(\ro u_i)\|_{L^2(Q_T)}  +  C  +\Big(C+\frac{1}{2}\int\limits_{Q_T} \chi^2 u_i^2 \partial_t \ro\,\mathrm{d}x \, \mathrm{d}t\Big)\\
   & \leq \frac12 \|\chi \nabla(\ro u_i)\|_{L^2(Q_T)}^2 + 8 \|\ro u_i\nabla \chi\|_{L^2(Q_T)}^2+ C + \frac{1}{2}\int\limits_{Q_T} \chi^2 u_i^2 \partial_t \ro\,\mathrm{d}x \, \mathrm{d}t \, ,
\end{align*}
where  we have also taken into account that $0\leq \chi^2 \ro u_if_i\leq C$ and used  Young's inequality.
Estimating  $\|\ro u_i\nabla \chi\|_{L^2(Q_T)}^2\leq C/(d')^2\leq C'$ then yields the bound
\begin{equation} 
\|\chi \nabla (\ro u_i)\|^2_{L^2(Q_T)}\leq C' + \underbrace{  \int\limits_{Q_T} \chi^2 u_i^2 \partial_t \ro\,\mathrm{d}x \, \mathrm{d}t}_{:=A}.
\label{eq:estimate_nabla(ro u_i)}
\end{equation}
We exploit now the structure of the system to control $A$. Indeed, since $\ro =\frac{\Phi(w)}{w}$ one easily computes for smooth positive solutions
$$
\partial_t\ro = \frac{d}{dw}\left(\frac{\Phi(w)}{w}\right)\partial_t w=\frac{w\Phi'(w)-\Phi(w)}{w^2}\big(\Delta (\ro w)+F\big).
$$
Thus integrating by parts gives
\begin{align*} 
A	& =\int\limits_{Q_T}\chi ^2 (\ro u_i)^2\frac{1}{(\ro w)^2}\big(w\Phi'(w)-\Phi(w)\big)\, \big(\Delta(\ro w)+F\big) \,\mathrm{d}x \, \mathrm{d}t\\
  & =-\int\limits_{Q_T}\nabla(\chi ^2) (\ro u_i)^2  \frac{1}{(\ro w)^2}  \big(w\Phi'(w)-\Phi(w)\big) \cdot \nabla(\ro w)\,\mathrm{d}x \, \mathrm{d}t\\
  & \phantom{=} -\int\limits_{Q_T}\chi ^2 \nabla(\ro u_i)^2  \frac{1}{(\ro w)^2}  \big(w\Phi'(w)-\Phi(w)\big)  \cdot \nabla(\ro w)\,\mathrm{d}x \, \mathrm{d}t\\
  & \phantom{=} -\int\limits_{Q_T}\chi ^2 (\ro u_i)^2  \nabla\left(\frac{1}{(\ro w)^2}\right)  \big(w\Phi'(w)-\Phi(w)\big)  \cdot \nabla(\ro w)\,\mathrm{d}x \, \mathrm{d}t\\
    & \phantom{=} -\int\limits_{Q_T}\chi ^2 (\ro u_i)^2  \frac{1}{(\ro w)^2}  \nabla\big(w\Phi'(w)-\Phi(w)\big) \cdot \nabla(\ro w) \,\mathrm{d}x \, \mathrm{d}t \\
   & \phantom{=}+\int\limits_{Q_T}\chi ^2 (\ro u_i)^2\frac{1}{(\ro w)^2}\big(w\Phi'(w)-\Phi(w)\big)F\,\mathrm{d}x \, \mathrm{d}t\\
    & =A_1 + A_2 + A_3 + A_4+B.
\end{align*}

Observing that $0\leq \ro u_i\leq \ro w =\Phi(w)$ and that  the hypothesis
\eqref{eq:structural_diffusion_hyp} implies that $ 0\leq w\Phi'(w)-\Phi(w)\leq C(a)\Phi(w)$,
 we can control the first term as
 \begin{align*}
A_1	& 	=-2\int\limits_{Q_T}\chi\nabla\chi\,\cdot (\ro u_i)^2 \frac{1}{(\ro w)^2} \big(w\Phi'(w)-\Phi(w)\big) \nabla(\ro w)\,\mathrm{d}x \, \mathrm{d}t\\
  & \leq C \|\nabla \chi\|_{L^2(Q_T)} \|\nabla (\ro w)\|_{L^2(Q_T)}\leq C'.
 \end{align*}
 \item
The second term is  bounded similarly as follows
\begin{align*}
A_2	& 	= -\, 2\int\limits_{Q_T}\chi^2 (\ro u_i)\nabla(\ro u_i)\,\cdot \frac{1}{(\ro w)^2} \big(w\Phi'(w)-\Phi(w)\big) \nabla(\ro w)\,\mathrm{d}x \, \mathrm{d}t\\
  & \leq 2\, \|\chi\nabla (\ro u_i)\|_{L^2(Q_T)} \, \left\|\chi\frac{\ro u_i}{\ro w}\frac{w\Phi'(w)-\Phi(w)}{\ro w}\nabla (\ro w)\right\|_{L^2(Q_T)}\\
  & \leq C\,  \|\chi\nabla (\ro u_i)\|_{L^2(Q_T)}\, \|\nabla (\ro w)\|_{L^2(Q_T)}\\
  & \leq \frac{1}{2}\, \|\chi\nabla (\ro u_i)\|_{L^2(Q_T)}^2	+	C \|\nabla (\ro w)\|_{L^2(Q_T)}^2\\
&\leq \frac{1}{2}\, \|\chi\nabla (\ro u_i)\|_{L^2(Q_T)}^2+C \, ,
\end{align*}
where we have also used the Young's inequality (the first term on the right-hand side will be reabsorbed into \eqref{eq:estimate_nabla(ro u_i)}).
The third quantity is controlled as
\begin{align*}
A_3	 &	= 2\, \int\limits_{Q_T}\chi^2 (\ro u_i)^2\frac{\nabla (\ro w)}{(\ro w)^3} \, \cdot  \big(w\Phi'(w)-\Phi(w)g\big) \nabla(\ro w) \,\mathrm{d}x \, \mathrm{d}t\\
& = 2\int\limits_{Q_T}\chi^2\frac{(\ro u_i)^2}{(\ro w)^2} \frac{w\Phi'(w)-\Phi(w)}{\Phi(w)}|\nabla (\ro w)|^2 \mathrm{d}x \, \mathrm{d}t  \leq C \|\nabla (\ro w)\|_{L^2(Q_T)}^2\leq C.
\end{align*}
In the fourth term we  write $\nabla(\ro w)=\nabla \Phi(w)=\Phi'(w)\nabla w$ and use \eqref{eq:convexity_hyp} to get
\begin{align*}
A_4	& = -\int\limits_{Q_T}  \chi^2 (\ro u_i)^2 \frac{1}{(\ro w)^2} \nabla \big(w\Phi'(w)-\Phi(w)\big)\,\cdot \nabla(\ro w) \,\mathrm{d}x\, \mathrm{d}t\\
  & = -\int\limits_{Q_T}  \chi^2 (\ro u_i)^2 \frac{1}{(\ro w)^2} w\Phi''(w)\nabla w\,\cdot \nabla(\ro w) \,\mathrm{d}x\, \mathrm{d}t\\
  & = -\int\limits_{Q_T}  \chi^2 (\ro u_i)^2 \frac{1}{(\ro w)^2} \frac{w\Phi''(w)}{\Phi'(w)}|\nabla(\ro w)|^2 \,\mathrm{d}x\, \mathrm{d}t\\
  & \leq \frac{1}{a}\|\nabla(\ro w)\|_{L^2(Q_T)}\leq C.
\end{align*}
\begin{remark}
Note that an upper bound  for $A_4$ is obtained here by  using the lower bound \eqref{eq:convexity_hyp} for $\frac{s\Phi''(s)}{\Phi'(s)}$. If in particular  $\Phi''(s)\geq 0$ for all $s\geq 0$, which is the typical case for the PME nonlinearity $\Phi(s)=|s|^{m-1}s$ in the range $m\geq 1$, then $A_4\leq 0$. This convexity condition is also valid for the Freundlich isotherm since $\beta_f(r)=\phi r+(1-\phi)r^p$ is concave and thus $\Phi_f=\beta_f^{-1}$ is convex.
\end{remark}
For the last term we obtain
$$
B	 =\int\limits_{Q_T}\chi ^2 (\ro u_i)^2\frac{1}{(\ro w)^2}\big(w\Phi'(w)-\Phi(w)\big)F  \,\mathrm{d}x\, \mathrm{d}t
\leq 
C(a)\int\limits_{Q_T}\Phi(w) F \,\mathrm{d}x\, \mathrm{d}t
\leq C.
$$
Plugging the above estimates back into \eqref{eq:estimate_nabla(ro u_i)} finally yields
$$
\|\nabla (\ro u_i)\|^2_{L^2(Q'_T)}\leq \|\chi \nabla (\ro u_i)\|^2_{L^2(Q_T)}\leq C'.
$$
Keeping track of the dependence of the estimates on $d'=\di\left(\overline{\Omega'},\partial\Omega\right)$ and optimizing all inequalities, one easily sees that  $C'=C(1+1/d')$ with $C=C(a,T,n,N,M)$ only and the proof is complete.
\end{proof}

We will next  address the regularity issue.
\begin{prop}
Let $\mathbf{u}$ and $M$ be as in  Proposition~\ref{prop:exists_classical_solutions}. There exist $\alpha=\alpha(a,n)\in (0,1)$ and  $C=C(a,T,n,N,M)>0$  such that the estimate
$$
\|\mathbf{u}\|_{C^{\alpha,\alpha/2}(Q')}\leq C(1+1/d'+1/\sqrt{\tau}),
$$
holds for any parabolic subdomain $Q'=\Omega'\times(\tau,T)$, where $0<\tau<T$, $\Omega'\subset\subset\Omega$, and $d'=\di(\overline{\Omega'},\partial\Omega)$.
\label{prop:C_alpha_estimate}
\end{prop}
%
% \textcolor{red}{Check again the dependence of the H\"older norm on the ``distance to the boundary'' $\tau,d'$. It cannot be simply $1/d_p$ (parabolic distance), because the estimate must blow up when either $\tau\to 0$ for fixed $d'>0$, or $d'\to 0$ for fixed $\tau >0$ (this would imply uniform continuity up to the boundary thus also continuity of the initial and boundary data, which we are not assuming at this stage). But the parabolic distance $d_p$ does not go to zero in any of these situations... This dependence has to be in some sense ``independent in $d'\to 0$ and $\tau\to 0$, which I need later in the proof of Theorem~\ref{theo:free_boundaries}.
% }
\begin{remark}
We would like to stress  that our proof handles both the nondegenerate $\Phi'(0)>0$ and degenerate $\Phi'(0)=0$ cases in a unified framework.
\end{remark}
\begin{proof}
The proof goes in several steps and is based on the intrinsic scaling method, cf. \cite{DBb}. As usual, we can assume after translation that the intrinsic cylinders are centered at the origin $(x_0,t_0)=(0,0)$. We  write $w=|\mathbf{u}|_1$ and recall that $w$ is positive and bounded with the $L^{\infty}$ bounds depending on $M,T$ only. We begin by considering the scalar problem.
\\

\noindent {\bf Step 1: Alternatives.} Suppose  that $\sup_{Q_r ^\mu} w \leq \mu$ in an intrinsic cylinder
\[
Q_r ^\mu := B_r  \times (-\tau_r ^\mu ,0)\,, \qquad \tau_r ^\mu := \frac{\mu}{\Phi(\mu)} r ^2\,,
\]
and define
\[
\widetilde Q_r ^\mu := B_{3 r /4} \times (- \frac34 \tau_r ^\mu , -\frac12 \tau_r ^\mu) \subset Q_r^\mu.
\]
 Now consider the following two alternatives
\begin{equation} \label{eq:nondeg}
\left| \widetilde Q_{r }^\mu \cap \{w \leq \mu/2\}  \right| \leq \delta |  \widetilde Q_{r }^\mu| \, ,
\end{equation}
\begin{equation} \label{eq:deg}
\left| \widetilde Q_{r }^\mu \cap \{w \leq \mu/2\}  \right| > \delta |  \widetilde Q_{r }^\mu| \, ,
\end{equation}
where $\delta \in (0,1)$ is a small parameter to be fixed shortly. 
The first one is the nondegenerate alternative and the second is the degenerate alternative. We will analyze them separately. 
\\

\noindent {\bf Step 2: Nondegenerate alternative 1.}  Set 
\[
r _j = \left(\frac1{4} + \frac1{4^{j+1}} \right) r \,, \qquad k_j := \left(\frac14 + \frac1{4^{j+1}} \right) \mu \,, \qquad w_j := (k_j-w)_+\,.
\]
Set also
\[
\widetilde Q^j := B_{r /4 + r _j} \times \left(- \frac12 \tau_r ^\mu - \tau_r ^\mu \left(\frac{r _j}{r }\right)^2 ,-\frac12 \tau_r ^\mu\right)\,,
\]
and let $\phi_j$ be a cut-off function such that $\phi_j$ is smooth, $0\leq \phi_j \leq 1$, $\phi_j$ vanishes on the parabolic boundary $\partial_p \widetilde Q^j$ (bottom and lateral), is one on $\widetilde Q^{j+1}$ and $|\nabla \phi_j|^2 + (\partial_t \phi_j^2)_+ \leq r ^{-2} 16^{j+1}$.     
Since $ w$ solves the equation $\partial_t w - \Delta \Phi(w) = F$ and $F\geq 0$, it is easy to check that $w_j$ is a weak subsolution to the equation $\partial_t w_j - \divo (\Phi'(w) \nabla w_j) = 0$, and testing the latter with $w_j \phi_j^2$ leads to the Caccioppoli inequality
\begin{eqnarray} %\label{eq:} 
\notag &&
\sup_{-\tau_r ^\mu<t<0} \mean{B_r }  \frac{w_j^2 \phi_j^2}{\tau_r ^\mu} \, dx +  \mean{Q_r ^\mu} \Phi'(w) |\nabla(w_j \phi_j)|^2\, dx \, dt 
\\ \notag && \qquad \qquad \leq c \mean{Q_r ^\mu} \left[ \Phi'(w) w_j^2 |\nabla \phi_j|^2 + w_j^2 (\partial_t \phi_j^2)_+ \right] \, dx \, dt \,. 
\end{eqnarray}
Setting
\[
\bar w_j := 
\begin{cases}
   k_j - k_{j+1}  & \mbox{if }w \leq k_{j+1}    \\
   w_j & \mbox{if } w > k_{j+1}\,  \\ 
\end{cases}
\]
and recalling from \eqref{eq:structural_diffusion_hyp} that $\Phi(s)/s$ is monotone non-decreasing with at most algebraic growth, we see that
\[
\frac{r ^2}{\tau_r ^\mu} |\nabla \bar w_j |^2 = \frac{\Phi(\mu)}{\mu} |\nabla \bar w_j |^2 \leq c(a) \frac{\Phi(k_{j+1})}{k_{j+1}} |\nabla \bar w_j |^2 \leq c(a) \Phi'(w) |\nabla w_j|^2
\]
since in the support of $\nabla \bar w_j$ we have $w \geq k_{j+1} \geq \mu/4$. Similarly,
\[
\Phi'(w) w_j^2 |\nabla \phi_j|^2 \leq c(a) 16^j \frac{1}{r ^2} \frac{\Phi(\mu)}{\mu}  w_j^2   \leq  \frac{c(a) 16^j}{\tau_r ^\mu}  \mu^2 \chi_{\{w<k_j\}}\,.
\]
Collecting estimates we arrive at
\[
\sup_{-\tau_r ^\mu<t<0} \mean{B_r }  \bar w_j^2 \phi_j^2 \, dx + r ^2  \mean{Q_r ^\mu} |\nabla (\bar w_j \phi_j) |^2 \, dx \, dt \leq c 16^j \mu^2\left( \frac{|\widetilde Q^j \cap \{w<k_j\}|}{|\widetilde Q^j| } \right)
\]
Next, the parabolic Sobolev embedding (see~\cite[Proposition 3.1, p.7]{DBb}) gives us 
\begin{eqnarray} %\label{eq:} 
\notag &&
\mean{Q_r ^\mu} (\bar w_j \phi_j)^{2(1+2/n)} \, dx \, dt 
\\ \notag && \qquad \leq c(n) \left( \sup_{-\tau_r ^\mu<t<0} \mean{B_r }  \bar w_j^2 \phi_j^2 \, dx + r ^2  \mean{Q_r ^\mu} |\nabla (\bar w_j \phi_j) |^2 \, dx \, dt \right)^{1+2/n}\,.
\end{eqnarray}
Since
\[
(\bar w_j \phi_j)^{2(1+2/n)} % \geq (k_j {-} k_{j+1})^{2(1+2/n)}  \chi_{\widetilde Q^{j+1} \cap \{w<k_{j+1}\}} 
\geq 4^{-6j} \mu^{2(1+2/n)}  \chi_{\widetilde Q^{j+1} \cap \{w<k_{j+1}\}}\,,
\]
we get 
\[
E_{j+1} \leq \bar c(a,n) 4^{8j} E_{j}^{1+2/n} 
\qquad \mbox{with} \quad
E_j :=  \frac{|\widetilde Q^j \cap \{w<k_j\}|}{|\widetilde Q^j| } \,.
\]
A standard iteration lemma on fast geometric convergence of series (\cite[p.12]{DBb}))  shows that if
$E_0 \leq \bar c^{-n/2} 4^{-2 n^2}$ then $E_j$ tends to zero as $j \to \infty$. Indeed, choosing $\delta := \bar c^{-n/2} 4^{-2 n^2} $, it follows from~\eqref{eq:nondeg} that $w \geq \mu/4$ in $B_{r /2} \times (-9 \tau_r ^\mu/16,- \tau_r ^\mu/2)$. 
\\

\noindent {\bf Step 3: Nondegenerate alternative 2.} 
We test $\partial_tw =\dive(\Phi'(w)\nabla w)+F$ with $(1/w -4/\mu)_+ \xi^2$, where $\xi \in C_0^\infty(B_{r /2})$, $0 \leq \xi \leq 1$,  $\xi \equiv 1$ in $ B_{r /4}$ and  $|\nablaÊ\xi| \leq 8/r $. Note that the chosen test function vanishes on $B_{r /2} \times \{- \tau_r ^\mu/2\}$ by Step 2. Taking advantage of $-F\left(\frac{1}{w}-\frac{4}{\mu}\right)_+\xi^2\leq 0$, straightforward manipulations then lead to
\begin{eqnarray} %\label{eq:} 
\notag &&
\sup_{-\tau_r ^\mu/2<t<0} \mean{B_{r /2}}  \log \left( \frac{\mu/4}{w} \right)_+ \xi^2 \, dx 
\\
\notag  && \qquad  + \frac{\tau_r ^\mu}{4} \mean{B_{r /2} \times (- \tau_r ^\mu/2,0)} \Phi'(w) \frac{|\nabla  (w-\mu/4)_+|^2}{w^2 } \xi^2 \, dx \, dt 
\\
\notag  && \qquad \qquad \leq  \tau_r ^\mu \mean{B_{r /2} \times (- \tau_r ^\mu/2,0)}  \Phi'(w)  |\nabla \xi|^2 \, dx \, dt + \mean{B_{r /2}}  \xi^2 \, dx 
\\
\notag  && \qquad \qquad \leq  \tau_r ^\mu \mean{B_{r /2} \times (- \tau_r ^\mu/2,0)}  \frac{1}{a}\frac{\Phi(w)}{w}  |\nabla \xi|^2 \, dx \, dt + \mean{B_{r /2}}  \xi^2 \, dx
\\
\notag  && \qquad \qquad \leq  \tau_r ^\mu \frac{1}{a} \frac{\Phi(\mu)}{\mu} \mean{B_{r /2} \times (- \tau_r ^\mu/2,0)}  |\nabla \xi|^2 \, dx \, dt + \mean{B_{r /2}}  \xi^2 \, dx\, \leq \frac{65}{a},
\end{eqnarray}
where we used successively \eqref{eq:structural_diffusion_hyp}, the monotonicity of $\Phi(s)/s$ with $w\leq \mu$, the definition $\tau_r ^\mu \frac{\Phi(\mu)}{\mu}=r^2$, and  the cutoff function properties $|\nabla\xi|\leq 8/r$, $\xi^2\leq 1$. As a consequence, we readily obtain
\[
\sup_{-\tau_r ^\mu/2<t<0} | B_{r /4} \cap \{ w(\cdot,t) < 4^{-1-m}\mu\} | \leq \frac1m \frac{2^n 65}{a \log 4} | B_{r /4} |\,,
\]
and in particular
\[
\frac{|B_{r /4} \times (- \tau_r ^\mu/2,0)  \cap \{w< 4^{-1-m}\mu\}|}{ |B_{r /4} \times (- \tau_r ^\mu/2,0)| } \leq \frac1m \frac{2^n 65}{a \log 4} \,  , 
\]
for any $m \in \en$. 

Next, redefine $k_j := 4^{-m-1} (2^{-1} + 2^{-1-j}) \mu$ and $r _j = (2^{-3}+2^{-3-j})r $, and set
\[
\hat Q^j := \hat B^j   \times (- \tau_r ^\mu/2,0)\,, \qquad \hat B^j := B_{r _j}(0)\,.
\]
Choose $\xi_j \in C_0^\infty(\hat B^j)$ in such a way that $0\leq \xi_j \leq 1$, $\xi_j = 1$ in  $\hat B^{j+1} $ and $|\nabla \xi_j| \leq 2^{4+j}/r $. 
The Caccioppoli estimate then takes the form
\[
\sup_{-\tau_r ^\mu/2<t<0} \mean{\hat B^j}  \frac{w_j^2 \xi_j^2}{\tau_r ^\mu} \, dx +  \mean{\hat Q^j } \Phi'(w) |\nabla(w_j \xi_j)|^2\, dx \, dt 
 \leq c \mean{\hat Q^j }  \Phi'(w) w_j^2 |\nabla \xi_j|^2\, dx \, dt \,,
\]
because $\xi_j$ is independent of time and the newly defined $w_j$ vanishes on the initial boundary of $\hat Q^j$ by Step 2. 
Since $s \mapsto \Phi(s)/s$ is a nondecreasing function and $1/\tau_r ^\mu = \Phi(\mu)/\mu$, it follows that
\[
\frac{\Phi(k_j)}{k_j} \sup_{-\tau_r ^\mu/2<t<0} \mean{\hat B^j}  w_j^2 \xi_j^2 \, dx +  \mean{\hat Q^j } \Phi'(w) |\nabla(w_j \xi_j)|^2\, dx \, dt \leq \hat c(a) \frac{4^j}{r ^2} \Phi(k_j) k_j \hat E_j\,,
\]
where this time $\hat E_j := | \hat Q^j \cap \{w_j>0\}|/ |\hat Q^j |$. Analogously to Step 2, we then arrive at $\hat E_{j+1} \leq c(n,a) 4^{8j} E_j^{1+2/n}$, and by choosing $m \equiv m(n,a)$
 large enough, i.e.
 \[
\hat E_0 \leq \frac1m \frac{2^n 65}{a \log 4} \leq \hat c^{-n/2} 4^{-2n^2}\,,
 \]
 we conclude that $w\geq 4^{-m-2} \mu$ in $Q_{r /8}^\mu$.
\\

% \noindent {\bf Step 4: Degenerate alternative.} 
% Let us then analyse the occurrence of~\eqref{eq:deg}. For this, set $v = \mu/2 - (w-\mu/2)_+$, which is a nonnegative weak supersolution to the equation 
% $\partial_t v - [\Phi(\mu)/\mu] \divo (b(x,t) \nabla v) = 0$ in $Q_\varrho^\mu$ with $b(x,t) := \mu \Phi'(w(x,t)) / \Phi(\mu)$. In the support of $\nabla v$ we have
% \[
% \nabla v(x,t) \neq 0  \qquad  \Longrightarrow \qquad c(a)^{-1} \leq  b(x,t) \leq c(a)\,.
% \]
% We redefine $b$ to be one on $\{\nabla v(x,t) = 0\}$. Scaling as $\bar b(x,t) = b(\varrho x, \tau_\varrho^\mu t)$ and $\bar v(x,t) = \mu^{-1} v(\varrho x, \tau_\varrho^\mu t)$, we see that $\bar v$ is a weak supersolution to the equation $\partial \bar v - \divo( \bar b \nabla \bar v) = 0$ in $B_1(0) \times (-1,0)$ and the coefficients $\bar b$ are measurable and bounded uniformly from below and from above by positive constants depending only on $a$. By the weak Harnack principle we then obtain
% \[
% \mean{B_{3/4} \times (-3/4,-1/2)} \bar v \, dx \leq c \inf_{B_{1/2} \times (-1/4,0)} \bar v\,.
% \]
% Scaling back to $v$ and recalling the definition of it, we actually obtain by~\eqref{eq:deg} that
% \[
% \sup_{Q_{\varrho/2}^\mu} w  \leq \mu - \frac1c \mean{\widetilde Q_{\varrho}^\mu } (\mu - w) \, dx \, dt \leq \mu - \frac{\delta}{2 c} \mu\,.
% \]

\noindent {\bf Step 4: Degenerate alternative.}
Let us then analyze the occurrence of~\eqref{eq:deg}. For this, set $v = \mu/2 - (w-\mu/2)_+ + \|F\|_{L^{\infty}(Q^{\mu}_r )}(t+\tau^{\mu}_r )$, which is a nonnegative weak supersolution to
$\partial_t v - [\Phi(\mu)/\mu] \divo (b(x,t) \nabla v) \geq 0 $ in $Q_r ^\mu$ with $b(x,t) := \mu \Phi'(w(x,t)) / \Phi(\mu)$. By definition of $v$ and $\mu=\sup w$ we have in the support of $\nabla v$
\[
\nabla v(x,t) \neq 0  \quad  \Rightarrow \quad  \mu/2 \leq w\leq \mu \quad \Rightarrow \quad c(a)^{-1} \leq  b(x,t) \leq c(a)\,.
\]
Redefining $b$ to be one on $\{\nabla v(x,t) = 0\}$ and scaling $b$ and $v$ as $\bar b(x,t) = b(r  x, \tau_r ^\mu t)$ and $\bar v(x,t) = v(r  x, \tau_r ^\mu t)$, we see that $\bar v$ is a weak supersolution to the equation $\partial \bar v - \divo( \bar b \nabla \bar v) = 0$ in $B_1 \times (-1,0)$ with measurable coefficient $\bar b$ bounded uniformly from below and from above by positive constants depending only on $a$. By the weak Harnack principle we then obtain
\[
\mean{B_{3/4} \times (-3/4,-1/2)} \bar v \, dx \leq c(a) \inf_{B_{1/2} \times (-1/4,0)} \bar v\,.
\]
Scaling back to $v$ and recalling its definition in terms of $w$, we finally get by \eqref{eq:deg}
% on the one hand
% \begin{align*}
% \mean{B_{3/4} \times (-3/4,-1/2)} \bar v	& =\mean{\tilde{Q}^\mu_r }v\\
% &	=\mean{\tilde{Q}^\mu_r } \mu/2 - (w-\mu/2)_+ +\mean{\tilde{Q}^\mu_r }|F|_{L^{\infty}(Q^{\mu}_r )}(t+\tau^{\mu}_r )\\
%   & \geq \delta\frac{\mu}{2}+|F|_{L^{\infty}(Q^{\mu}_r )}\frac{\tau^\mu_r }{8}
% \end{align*}
% by \eqref{eq:deg}, and on the other hand
% \begin{align*}
% \inf_{B_{1/2} \times (-1/4,0)} \bar v=  \inf\limits_{Q^\mu_{r /2}}\,v \leq \left(\mu-\sup\limits_{Q^\mu_{r /2}} \, w \right) +  |F|_{L^{\infty}(Q^{\mu}_r )} \tau^{\mu}_r .
% \end{align*}
% Plugging into the previous Harnack Inequality with gives
$$
\sup\limits_{Q^\mu_{r /2}} \, w\leq \mu\left(1-\frac{\delta}{2c}\right)  +  \tau^{\mu}_r   \|F\|_{L^{\infty}(Q^{\mu}_r )}\left(1-\frac{1}{8c}\right).
$$

\noindent {\bf Step 5: Conclusion from alternatives.} 
Taking 
\[\sigma =\sigma(a,n):= \min\left\{4^{-m-2},\frac{\delta}{4c}, \frac{1 - 16^{-a}}{2} \right\}\,\in(0,1),
\]
we deduce from the previous alternatives that 
\[
\mbox{either} \qquad \inf_{Q_{r/8}^{\mu}} w \geq \sigma \mu \qquad \mbox{or} \qquad \sup_{Q_{r/2}^{\mu}} w \leq (1-2\sigma) \mu + 
\tau^{\mu}_r  |F|_{L^{\infty}(Q^{\mu}_r)}
\]
holds provided that $\sup_{Q_{r}^{\mu}} w \leq \mu$. Choose now any $R>0$ such that $B_R\times(-R^2,0)\subset Q_T$ (after translation), let $R_j = 8^{-j} R$,
\[
\mu_0 := 1 + \Phi^{-1} \left( \sigma^{-1} R^2 \|F\|_{L^{\infty}(B_R  \times (-R^2,0 )} \right) + \sup_{B_R  \times (-R^2,0 )}w\,,
\]
and then inductively 
\[
 \mu_{j+1} := (1-2\sigma) \mu_{j}  +  \frac{\mu_{j}}{\Phi(\mu_{j})} R_j^2 \|F\|_{L^{\infty}(B_R  \times (-R^2,0 )}
\]
for $j \geq 0$. Clearly $\mu_j \geq (1-2 \sigma)^j \mu_0$. Using the algebraic growth \eqref{eq:algebraic_growth_Phi(s)/s} leads to
\begin{equation}
\begin{array}{l}
\displaystyle \frac{R_j^2 \|F\|_{L^{\infty}(B_R  \times (-R^2,0 )} }{\Phi(\mu_{j})}  =  \frac{\Phi(\mu_0)}{\Phi(\mu_{j})} R_j^2 \frac{\|F\|_{L^{\infty}(B_R  \times (-R^2,0 )}}{\Phi(\mu_0)}
\\ \\ \displaystyle \qquad \qquad   \leq  \left( \frac{(1-2\sigma)^{-1/a}}{65} \right)^j \frac{R^2 \|F\|_{L^{\infty}(B_R  \times (-R^2,0 )}}{\Phi(\mu_0)} \leq  \sigma,  
\end{array}
 \label{eq:sigma_bnd} 
\end{equation}
where the last inequality follows from the definition of $\sigma$ and from the bound
\[
\mu_0\geq \Phi^{-1} \left( \sigma^{-1} R^2 \|F\|_{L^{\infty}}\right) \quad \Longrightarrow \quad \Phi(\mu_0)\geq \sigma^{-1} R^2 \|F\|_{L^{\infty}}\, .
\]
 Similarly, we get
\[
\frac{\mu_{j+1}}{\Phi(\mu_{j+1})} \left(\frac{R_j}{8}\right)^2 = \frac1{16} \frac{\mu_{j+1}}{\mu_{j}} \frac{\Phi(\mu_{j})}{\Phi(\mu_{j+1})} \frac{\mu_{j}}{\Phi(\mu_{j})}  \left(\frac{R_j}{2}\right)^2 \leq \frac{\mu_{j}}{\Phi(\mu_{j})}  \left(\frac{R_j}{2}\right)^2\,.
\]
Therefore, we conclude that 
\[
(1-2\sigma) \mu_j \leq \mu_{j+1} \leq (1-\sigma) \mu_j \qquad \mbox{and} \qquad Q_{R_{j+1}}^{\mu_{j+1}} \subset Q_{R_{j}/2}^{\mu_{j}} \,,
\]
and alternatives reduce to
\[
\mbox{either} \qquad \inf_{Q_{R_{j+1}}^{\mu_j}} w \geq \sigma \mu_j \qquad \mbox{or} \qquad \sup_{Q_{R_{j+1}}^{\mu_{j+1}}} w \leq \mu_{j+1}\,,
\]
provided that $\sup_{Q_{R_{j}}^{\mu_{j}}} w \leq \mu_j$. Observe that $\sup_{Q_{R_{0}}^{\mu_{0}}} w \leq \mu_0$ by the definition of~$\mu_0$. Since we are considering positive solutions $w>0$ and $\mu_j \to 0$ as $j\to \infty$ the degenerate alternative can clearly occur at most a finite number of times, thus
\[
\sigma \mu_J \leq \inf_{Q_{R_J/8}^{\mu_J }} w \leq \sup_{Q_{R_J/8}^{\mu_J }} w \leq \mu_J\,\quad \mbox{for some finite } J.
\]
Since $\sigma=\sigma(a,n)$ only, the algebraic growth \eqref{eq:algebraic_growth_Phi(s)/s} then readily implies that
\[
\frac1{c(a,n)} \frac{\Phi(\mu_J)}{\mu_J} \leq \frac{\Phi(w)}{w} \leq c(a,n) \frac{\Phi(\mu_J)}{\mu_J} \qquad \mbox{in} \quad  Q_{R_J/8}^{\mu_J }\,.
\]
Scaling as in step 4 and writing $\partial_t u_i=\Delta\ro u_i+f_i=\dive(\ro\nabla u_i)+(\ldots)$, where $\ro=\Phi(w)/w$, we see that each scaled component $\overline{u}_i$ solves in $B_{1/8}\times(-1,0)$ a uniformly parabolic linear equation $\partial_t \overline{u}_i=\dive (\overline{\ro}\nabla \overline{u}_i)+(\ldots)$ in divergence form with a measurable coefficient $\overline{\ro}$ satisfying $c(a,n)\leq \overline{\ro}\leq c(a,n)^{-1}$. In particular, $\overline{u}_i$ are H\"older continuous and satisfy a DeGiorgi-Nash-Moser oscillation estimate which, scaling back to $u_i$, takes the explicit form
\[
\osc_{Q_{\theta R_J}^{\mu_J }} u_i \leq c \theta^\beta \mu_J + c \frac{\mu_J}{\Phi(\mu_J)} (\theta R_J)^2 \|f_i\|_{L^{\infty}\left(B_R  \times (-R^2,0 )\right)} \, ,
%\leq  c ( \theta^\beta + \theta^2) \mu_J
\] 
for some $\beta=\beta(a,n)$, $c=c(a,n)$ only and for all $\theta \in (0,1)$. Observing that $0\leq f_i\leq |\mathbf{f}|_1=F$ and  recalling estimate \eqref{eq:sigma_bnd} which holds for  all $j$ and with $\sigma=\sigma(a,n)$, we obtain
\[
\forall \,\theta\in(0,1): \qquad \osc_{Q_{\theta R_J}^{\mu_J }} u_i \leq  c\, ( \theta^\beta + \theta^2) \mu_J.
\] 
Setting 
\[\alpha =\alpha(a,n):= \min\{\beta,- \log(1-\sigma)/\log 8,2\}\] and increasing the constant $c$ by a factor depending only on $n,a$, we   conclude that
\[
\osc_{B_r \times (-r^2,0)} u_i \leq c \left( \frac{r}{R} \right)^\alpha \left( \frac{\Phi(\mu_0)}{\mu_0} \right)^{\alpha/2} \mu_0
\]
for all $r \in (0,R]$. This yields the desired interior H\"older continuity estimate after standard manipulations.
\end{proof}
% Note that in the above proof we need in any case to iterate sufficiently many times so that the nondegenerate alternative holds in $Q^{\mu_J}_{R_J}$ after a finite number number of steps, but this may in fact require many steps. Since each step reduces the size of the nested (intrinsic) cylinders. As a consequence we really have to step away from the parabolic boundary in order to guarantee the nondegeneracy at some point, hence
%
Assuming regularity and compatibility from the data, the solution can be shown to be  H\"older continuous up to the boundary.
\begin{prop}
Let $\mathbf{u}$, $M$ be as in Proposition~\ref{prop:exists_classical_solutions} and $\alpha=\alpha(a,n)$ as in Proposition~\ref{prop:C_alpha_estimate}. Assume further that the initial and boundary data are compatible and $\beta$-H\"older continuous with some $\beta\in(0,1)$, i.e. there is $\mathbf{U}\in\mathcal{C}^{\beta,\beta/2}(\overline{Q^T})$ such that $\mathbf{u}^0=\mathbf{U}(\,.\,,0)$ and $\mathbf{z}^D=\frac{\Phi(|\mathbf{U}|_1)}{|\mathbf{U}|_1}\mathbf{U}$ in $\Sigma_T$. Then $\mathbf{u}\in\mathcal{C}^{\gamma,\gamma/2}(\overline{Q^T})$, with $\gamma=\min(\alpha,\beta)$. Moreover, the H\"older norm  of $\mathbf{u}$ depends only on $a,n,N,M,T$ and on the $\beta$-H\"older norm of the data.
\label{prop:boundary_regularity}
\end{prop}
\begin{proof}
Our assumptions on the data turn into similar compatibility and regularity conditions for the scalar problem \eqref{eq:PB_w}. A straightforward modification of our interior argument (see, e.g.,  \cite{DBb,Va07}) shows that $w$ is $\mathcal{C}^{\gamma,\gamma/2}$ up to the boundary, which in turn yields  the same regularity for $u_i$ through the linear parabolic equation $\partial_t u_i=\Delta(\ro u_i)+f_i=\dive(\ro \nabla u_i)+\dive(u_i\nabla\ro)+f_i$.
%\textcolor{red}{Huuumm... I would honestly not be convinced if I was reading the paper!}
\end{proof}

%!TEX encoding = UTF-8 Unicode
%%%%%%%%%%%%%%%%%%%%%%%%%%%%%%%%
\section{Weak solutions}
\label{section:weak_sols}
Let us first introduce different notions of solutions.
\begin{definition}[weak solutions] \

\begin{enumerate}
 \item[(i)]
A non-negative function $ w \in L^{\infty}(Q_T)$ is called bounded very weak solution of \eqref{eq:PB_w} if the equality
\begin{equation}
\qquad \int\limits_{Q_T} \{w\partial_t \ff  + \Phi(w) \Delta \ff  + F \ff \} \,\mathrm{d}x\, \mathrm{d}t
= - \int\limits_{\Omega} w^0(x) \ff (x,0) \,\mathrm{d}x  + \int\limits_{\Sigma_T} g^D\frac{\partial \ff }{\partial \nu} \,\mathrm{d}x\, \mathrm{d}t
 \label{eq:weak_form_w}
\end{equation}
holds for all $\ff \in\mathcal{C}^{2,1}(\overline{Q_T})$ vanishing on $\Sigma_T$ and in $\Omega\times\{t=T\}$.
\label{def:very_weak_solution_w}
\item[(ii)]
A  non-negative function $w \in L^{\infty}(Q_T)$ is called a bounded weak energy solution of \eqref{eq:PB_w} if $\Phi(w)\in L^2(0,T;H^1(\Omega))$, the trace $\gamma(\Phi(w))=g^D$ in $L^{2}(0,T;H^{1/2}(\partial\Omega))$ and the equality
\begin{equation}
\qquad \int\limits_{Q_T} \{w\partial_t \ff  - \nabla \Phi(w) \cdot \nabla \ff  + F \ff \}\,\mathrm{d}x\mathrm{d}t
= - \int\limits_{\Omega} w^0(x) \ff (x,0) \,\mathrm{d}x
 \label{eq:energy_weak_form_w}
\end{equation}
holds for all $\ff \in\mathcal{C}^{2,1}(\overline{Q_T})$ vanishing on $\Sigma_T$ and in $\Omega\times\{t=T\}$.
\label{def:weak_energy_solution_w}
\item[(iii)]
A function $\mathbf{u}=(u_1,\ldots,u_N)\in L^{\infty}(Q_T)$ is called a (non-negative) bounded very weak solution of \eqref{eq:PB_u} if $u_i\geq 0$ a.e. in $Q_T$ and the equality
\begin{equation}
\int\limits_{Q_T} \{ u_i \partial_t \ff  + \varrho  u_i \Delta \ff  + f_i \ff \}\,\mathrm{d}x\mathrm{d}t\, 
+ \int\limits_{\Omega} u_i^0(x) \ff (x,0)\,\mathrm{d}x = \int\limits_{\Sigma_T} z^D_i\frac{\partial \ff }{\partial \nu}\,\mathrm{d}x\mathrm{d}t,
 \label{eq:weak_form_u}
\end{equation}
where $\varrho  =\frac{\Phi(|\mathbf{u}|_1)}{|\mathbf{u}|_1}$, holds for any $ i=1,\ldots,N$ and for all  $\ff \in\mathcal{C}^{2,1}(\overline{Q_T})$ vanishing on $\Sigma_T$ and  in  $\Omega\times\{t=T\}$.
\label{def:very_weak_solution_u}
\item[(iv)]
A function $\mathbf{u} \in L^{\infty}(Q_T)$ is called a (non-negative) bounded weak energy solution of \eqref{eq:PB_u} if $u_i\geq 0$ a.e. in $Q_T$, $(\ro u_i)\in L^2(0,T;H^1(\Omega))$, the trace $\gamma(\ro u_i)=z^D_i$  in $L^{2}(0,T;H^{1/2}(\partial\Omega))$, and the equality
\begin{equation}
\int\limits_{Q_T} \{w\partial_t \ff  - \nabla (\ro u_i) \cdot \nabla \ff  + f_i \ff \}\,\mathrm{d}x\mathrm{d}t
= - \int\limits_{\Omega} z_i^0(x) \ff (x,0)\,\mathrm{d}x,
 \label{eq:energy_weak_form_u}
\end{equation}
where $\ro=\frac{\Phi\left(|\mathbf{u}|_1\right)}{|\mathbf{u}|_1}$, holds for any  $ i=1,\ldots,N$ and for all $\ff \in\mathcal{C}^{2,1}(\overline{Q_T})$ vanishing on $\Sigma_T$ and in $\Omega\times\{t=T\}$.
\label{def:weak_energy_solution_u}
\end{enumerate}
\label{def:weak_solutions}
\end{definition}
The notion of very weak solutions preserves the diagonal structure of the system as shown in the following lemma. 
\begin{lemma}
If $\mathbf{u}$ is a non-negative bounded very weak (resp. energy) solution of system \eqref{eq:PB_u} in the sense of Definition~\ref{def:weak_solutions} then $w=|\mathbf{u}|_1$ is a non-negative bounded very weak (resp. energy) solution to problem \eqref{eq:PB_w} in the sense of Definition~\ref{def:weak_solutions}.
\label{lem:u_solution=>w_solution}
\end{lemma}
\begin{proof}
Sum equalities \eqref{eq:weak_form_u} over $i$ from $1$ to $N$ and observe that by definition $\sum_i \varrho  u_i=\varrho  \sum_i u_i=\ro |\mathbf{u}|_1=\Phi(w)$, $\sum_i f_i=F$, $\sum_i u^0_i=|\mathbf{u}^0|_1=w^0$, and $\sum_i z^D_i=|\mathbf{z}^D|_1=g^D$.
\end{proof}
We will now address uniqueness. Note that the following proposition guarantees uniqueness also within the class of energy solutions since weak energy solutions are in particular very weak solutions.
\begin{prop}[Uniqueness]
Given the non-negative and bounded data $\mathbf{f},\mathbf{u}^0,\mathbf{z}^D$ there exists at most one non-negative bounded very weak solution to problem~\eqref{eq:PB_u} in the sense of Definition \ref{def:weak_solutions}. 
\label{prop:uniqueness_weak_sols}
\end{prop}

\begin{proof}
Let $\mathbf{u}^1$ and $\mathbf{u}^2$ be two solutions to problem~\eqref{eq:PB_u}, corresponding to the same initial and boundary data. It follows from the previous lemma  that $w^1=|\mathbf{u}^1|_1$ and $w^2=|\mathbf{u}^2|_1$ are both bounded very weak solutions to the same Cauchy-Dirichlet problem \eqref{eq:PB_w}.  A standard comparison result for such solutions \cite[Theorem 6.5]{Va07} provides uniqueness within this class. Thus $w^1=w^2=w$ and, in particular, the pressures coincide,  $\varrho ^1=\varrho ^2=\varrho =\frac{\Phi(w)}{w}$.  
Next, we use a duality proof, as in proving  the comparison results for  GMPE, to show that $\mathbf{u}^1=\mathbf{u}^2$. In fact, the situation here  is  simpler because we already know that $\varrho ^1=\varrho ^2$. For the sake of completeness, we   nonetheless give the details.
\par
Fixing any $i\in 1\ldots N$, denoting $\tilde{u}=u^1_i-u^2_i$, and subtracting the weak formulation \eqref{eq:weak_form_u} satisfied by $u^2$ from that satisfied by $u^1$, we see that
\begin{equation}
\int\limits_{Q_T} \{\tilde{u} \partial_t \ff  + \varrho  \tilde{u} \Delta \ff\}\,\,\mathrm{d}x\mathrm{d}t   =0
\label{eq:weak_u1-u2}
\end{equation}
for all $\ff \in \mathcal{C}^{2,1}(\overline{Q_T})$ vanishing on $\Sigma_T\cup\{t=T\}$. Fix some arbitrary $\theta\in \mathcal{C}^{\infty}_c(Q_T)$, choose $\eps>0$, and let $\varrho _{\eps}=\max\{\varrho , \eps\}$. Since $\mathbf{u}^1$ and $\mathbf{u}^2$ are bounded so is $\varrho =\varrho ^1=\varrho ^2$, and we can construct a smooth approximation $\{\varrho _{\eps,k}\}_{k\in \mathbb{N}}$ to $\varrho _{\eps}$ such that $\eps\leq \varrho _{\eps,k}\leq C$. For fixed $\eps,k$ we can then solve the approximate dual backward equation
\begin{equation}
\left\{
\begin{array}{ll}
\partial_t \ff  +\varrho _{\eps,k} \Delta \ff  = \theta \qquad & \text{in }Q_T\\
 \ff  =0	&	\text{in }\Sigma_T\\
 \ff (.,T) =0		&	\text{in }\Omega.
\end{array}
\right.
\label{eq:dual}
\end{equation}
for a unique $\ff =\ff_{\eps,k}\in \mathcal{C}^{2,1}(\overline{Q_T})\cap \mathcal{C^{\infty}}(Q_T)$. Since $\ff $ vanishes by construction on  $\Sigma_T$ and in $\Omega\times\{t=T\}$ it is admissible as a test function in \eqref{eq:weak_u1-u2}. This gives
\begin{align*}
\left|\int_{Q_T}\tilde{u}\theta \,\mathrm{d}x\mathrm{d}t \right|	&=	\left|\int_{Q_T} \tilde{u} (\varrho -\varrho _{\eps,k}) \Delta\ff \,\mathrm{d}x\mathrm{d}t\right|\\
&\leq  \left(\, \int \limits_{Q_T} \tilde{u}^2\frac{|\varrho  - \varrho _{\eps,k}|^2}{\varrho _{\eps,k}} \,\mathrm{d}x\mathrm{d}t \right)^{1/2}
\left(\, \int\limits_{Q_T}\varrho _{\eps,k} |\Delta \ff |^2 \,\mathrm{d}x\mathrm{d}t \right)^{1/2}\\
  & \leq \frac{C}{\eps^{1/2}}\|\varrho -\varrho _{\eps,k}\|_{L^2(Q_T)} \left(\, \int\limits_{Q_T}\varrho _{\eps,k} |\Delta \ff |^2 \,\mathrm{d}x\mathrm{d}t \right)^{1/2},
\end{align*}
because $\tilde{u}\in L^{\infty}$ and $\varrho _{\eps,k}\geq \eps$. Since $\ff $ is smooth,  a straightforward computation shows that (cf. \cite[Theorem 6.5]{Va07})
$$
\left(\int\limits_{Q_T}\varrho _{\eps,k} |\Delta \ff |^2\,\mathrm{d}x\mathrm{d}t\right)^{1/2}\leq C \|\nabla \theta\|_{L^2(Q_T)}
$$
for some $C>0$ independent of $\eps,k,\theta$. For fixed $\eps>0$ we can then choose $k$ large enough such that $|\varrho _{\eps}-\varrho _{\eps,k}|_{L^2(Q_T)}\leq \eps$. By definition of the cutoff function $\varrho _{\eps}$, we have $0\leq \ro_{\eps}-\ro\leq\eps$. Hence $|\varrho  -\varrho _{\eps}|_{L^2(Q_T)}\leq \eps|Q_T|^{1/2}$ so that $|\varrho -\varrho _{\eps,k}|_{L^2(Q_T)}\leq |\varrho -\varrho _{\eps}|_{L^2(Q_T)}+|\varrho _{\eps}-\varrho _{\eps,k}|_{L^2(Q_T)}\leq C\eps$, and we obtain
$$
\left|\int_{Q_T}\tilde{u}\theta\,\mathrm{d}x\mathrm{d}t\right|\leq C\eps^{1/2}\|\nabla \theta\|_{L^2(Q_T)}.
$$
Because $\theta\in \mathcal{C}^{\infty}_c(Q_T)$ was arbitrary and $\eps$ was independent of $\theta$ we conclude letting $\eps\to 0$ that $\tilde{u}=u^1_i-u^2_i=0$ a.e. in $ Q_T$ and the proof is complete.
\end{proof}
\begin{remark}
The above uniqueness  proof does not really require $L^{\infty}(Q_T)$ bounds but merely that $\mathbf{u},\varrho  \mathbf{u}\in L^{2}_{loc}(Q_T)$. In fact, scalar parabolic equations such as \eqref{eq:PB_w} benefit usually  from  smoothing properties that should allow one to extend the theory  to $L^{1}$ data. Due to the coupled vectorial nature of the problem and the lack of space we shall not pursue  this direction here.
\end{remark}
%

% \begin{theorem}
% For any non-negative bounded data $\mathbf{u}^0,\mathbf{z}^D,\mathbf{f}$ there exists a unique non-negative bounded very weak solution $\mathbf{u}$ to \eqref{eq:PB_u}. Moreover, $w=|\mathbf{u}|_1$ is a non-negative bounded very weak solution to \eqref{eq:PB_w} and there exist positive constants $\alpha=\alpha(a,n)\in (0,1)$  and  $C=C(a,T,n,N,M)$ such that
% \begin{equation}
% \|\mathbf{u}\|_{C^{\alpha,\alpha/2}(Q')}\leq C(1+1/d'+1/\sqrt{\tau})
% \label{eq:uniform_Holder}
% \end{equation}
% holds in all parabolic subdomains  $Q'=\Omega'\times(\tau,T)$, with $0<\tau<T$ and $d'=\di(\overline{\Omega'},\partial\Omega)$.
% 
% Finally if the data satisfy \eqref{hyp:boundary_data} and \eqref{hyp:initial+forcing}, then $w$ is a global weak energy solution to \eqref{eq:Cauchy_PB_w} and $\mathbf{u}$ is a local weak energy solution to \eqref{eq:Cauchy_PB_u} in the sense that
% \begin{equation}
%  \|\nabla(\ro u_i)\|_{L^2(Q'_T)} \leq C(1+1/d')\qquad  \forall\, i=1\ldots N\, , 
% \label{eq:energy_very_weak_ui}
% \end{equation}
% where $\ro=\frac{\Phi(w)}{w}=\frac{\Phi(|\mathbf{u}|_1)}{|\mathbf{u}|_1}$, holds in any $Q'_T=\Omega'\times(0,T)$ with some constant $C=C(a,T,n,N,M)>0$. 
% \label{theo:exist_sols_Dirichlet_Cauchy_PB}
% \end{theorem}
Theorem~\ref{theo:exist_sols_Dirichlet_Cauchy_PB} allows for the initial data $\mathbf{u}^0$ to vanish identically in some ball $B_r(x_0)\subset \Omega$. As we  will see in  Section~\ref{section:FB}, this leads to free boundaries in the degenerate case $\Phi'(0)=0$. It will  also become clear in the proof that the structural condition \eqref{eq:convexity_hyp} needs to be enforced only to get the estimate \eqref{eq:energy_very_weak_ui}. In fact, this energy estimate plays no role whatsoever in the analysis so one may actually dispense with it and limit oneself to very weak solutions of \eqref{eq:PB_u}.
\begin{proof}[Proof of Theorem~\ref{theo:exist_sols_Dirichlet_Cauchy_PB}]
Uniqueness follows from  Proposition~\ref{prop:uniqueness_weak_sols}. The existence argument is based on a ``lifting'' technique, classical for scalar GPME and working here thanks to the diagonal structure of the system.
\par
We first lift and approximate  the  bounded non-negative data $\mathbf{u}^0,\mathbf{f},\mathbf{z}^D$ component-wise by smooth functions   $u_i^{0,k},f_i^k$ and $z_i^{D,k}$ such that $\frac{1}{k}\leq u_i^{0,k},f_i^k,z_i^{D,k}\leq C+\frac{1}{k}$
for some constant $C>0$ depending only on the data, and
$$
\|\mathbf{u}^{0,k}-\mathbf{u}^0\|_{L^{1}(\Omega)} + \|\mathbf{f}^k-\mathbf{f}\|_{L^1(Q_T)}	+	\|\mathbf{z}^{D,k}-\mathbf{z}^{D}\|_{L^1(\Sigma_T)} \to 0
$$
 as $k\to\infty$. By Proposition~\ref{prop:exists_classical_solutions}, given the smooth data $\mathbf{u}^{0,k}, \mathbf{f}^k, \mathbf{z}^{D,k}$ there exists a  positive classical solution $\mathbf{u}^k$ to  \eqref{eq:PB_u} which is bounded in $Q_T$ uniformly in $k$. By virtue of Proposition~\ref{prop:C_alpha_estimate} $\{\mathbf{u}^k\}_{k}$ is also bounded in $\mathcal{C}^{\alpha,\alpha/2}(Q')$ for any subdomain $Q'$ and for some  $\alpha=\alpha(a,n)\in(0,1)$. By diagonal extraction we may then  assume that $\mathbf{u}^k\to \mathbf{u}$ in $\mathcal{C}_{\loc}(Q_T)$ with the limit function $\mathbf{u}$ satisfying the local $\mathcal{C}^{\alpha,\alpha/2}$-estimate \eqref{eq:uniform_Holder}.  In particular $\mathbf{u}^k(x,t)\to \mathbf{u}(x,t)\geq 0$ pointwise in $Q_T$. From the continuity of $\Phi$ with $\lim\limits_{s\to 0}\frac{\Phi(s)}{s}=\Phi'(0)$ it then follows that $\ro^k=\frac{\Phi(|\mathbf{u}^k|_1)}{|\mathbf{u}^k|_1)}\to \frac{\Phi(|\mathbf{u}|_1)}{|\mathbf{u}|_1)}=\ro$ a.e. in $Q_T$. 
Since $\mathbf{u}^k,\ro^k$ are bounded uniformly in $L^{\infty}(Q_T)$ 
we conclude by dominated convergence that $u^k_i\to u_i$ and $\ro^k u^k_i\to \ro u_i$ in $L^p(Q_T)$ for all $p\in [1,\infty)$. Given that $\mathbf{u}^k$ is a smooth positive solution and that for all $i=1,\ldots, N$ it holds
$$
 \int\limits_{Q_T}\{ u^k_i \partial_t \ff + \varrho^k  u^k_i \Delta \ff + f^k_i \ff \}\,\mathrm{d}x\mathrm{d}t  
+ \int\limits_{\Omega} u_i^{0,k}(x) \ff (x,0)\mathrm{d}x = \int\limits_{\Sigma_T} z^{D,k}_i\frac{\partial \ff }{\partial \nu}\mathrm{d}x\mathrm{d}t.
$$
The previous strong $L^p(Q_T)$ convergence and convergence of the data allow one to send $k\to\infty$ to obtain \eqref{eq:weak_form_u}. Similarly by Lemma~\ref{lem:u_solution=>w_solution} we see that $w=\lim w^k=\lim |\mathbf{u}^k|_1=|\mathbf{u}|_1$ is a very weak solution to \eqref{eq:PB_w}.

Regarding the energy estimates, if the data satisfy \eqref{hyp:boundary_data} and \eqref{hyp:initial+forcing}, then they can be approximated as before by smooth positive data satisfying in addition
\begin{eqnarray} \label{eq:} 
\notag
&& \|\mathbf{u}^{0,k}\|_{L^{\infty}(Q_T)} + \|\mathbf{f}^k\|_{L^{\infty}(Q_T)} 
\\ \notag
&& \qquad + \|\mathbf{z}^{D,k}\|_{L^{\infty}(Q_T)} + \|\mathbf{z}^{D,k}\|_{L^2(0,T;H^1(\Omega))} + \|\partial_t \mathbf{z}^{D,k}\|_{L^{\infty}(Q_T)}\leq C.
\end{eqnarray}
By Proposition~\ref{prop:energy_estimate} we get
$$
 \|\nabla (\ro^k w^k)\|_{L^2(Q_T)}\leq C
$$
and
$$
 \|\nabla (\ro^k u_i^k)\|_{L^2(Q'_T)}\leq C(1+1/d') \qquad \forall i=1\ldots N
$$
uniformly in $k$ for some $C=C(a,n,N,T,M)$ only. Since $\ro^k u_i^k,\ro^k w^k\to \ro u_i,\ro w$ we conclude that $\nabla(\ro u_i),\nabla(\ro w)$ satisfy the same $L^2$ bounds and the proof is complete.
\end{proof}

%
%%%%%%%%%%%%%%%%%%%5555
\section{Free boundaries}
\label{section:FB}
In this section we set $Q=\R^n\times (0,\infty)$, $Q^{\tau,T}=\R^n\times(\tau,T)$, $Q^T=\R^n\times(0,T)$, and consider the Cauchy Problem
\begin{equation}
\left\{
\begin{array}{ll}
 \partial_t \mathbf{u}=\Delta\left(\frac{\Phi(|\mathbf{u}|_1)}{|\mathbf{u}|_1}\mathbf{u}\right)	 & \mbox{in } \, \R^n\times(0,\infty)\\
 \mathbf{u}(x,0) =  \mathbf{u}^0 (x)	& \mbox{in }\R^n
\end{array}
\right.
 \label{eq:Cauchy_PB_u}
\end{equation}
with a non-negative, bounded and compactly supported initial data $\mathbf{u}^0 $. If the modulus of ellipticity $\ro=\frac{\Phi(|\mathbf{u}|_1)}{|\mathbf{u}|_1}$ in \eqref{eq:Cauchy_PB_u} vanishes  when $|\mathbf{u}|_1=0$ (the degenerate case), compactly supported solutions  should evolve  from compactly supported initial data. By analogy with scalar equations, the free boundary $\Gamma(t):=\partial \operatorname{supp}\mathbf{u}(\,.\,,t)$ should then  propagate with finite speed in the sense that, at any $x_0\notin \operatorname{supp}\,\mathbf{u}(\,.\,,t_0)$ we should have $x_0\notin \operatorname{supp}\,\mathbf{u}(\,.\,,t_0+h)$ for small enough $h>0$. Although this behaviour is  well understood for scalar equations, the coupled nature of  system \eqref{eq:Cauchy_PB_u} prevents us from just recalling known results. Instead, we will  again resort  to our central idea,  based on the particular structure of the system, that 
controlling $w=|\mathbf{u}|_1$ controls each individual species $u_i$. Indeed, $0\leq u_i\leq |\mathbf{u}|_1=w$, thus $\operatorname{supp}\,u_i\subset \operatorname{supp}\,w$ and $\mathbf{u}$ will  propagate with finite speed as long as $w$ does. This will in turn be ensured by looking at the scalar problem 
\begin{equation}
\left\{
\begin{array}{ll}
 \partial_t w=\Delta\Phi(w)	 & \mbox{in } \, \R^n\times(0,\infty)\\
 w(x,0) =  w^0 (x)	& \mbox{in }\R^n
\end{array}
\right.
 \label{eq:Cauchy_PB_w}
\end{equation}
with a non-negative, bounded and compactly supported initial data $w^0=|\mathbf{u}^0|_1$. 

Given that assumption \eqref{eq:structural_diffusion_hyp} does not rule out  nondegenerate diffusion (we may have $\Phi'(0)>0$) for which the finite speed of propagation obviously fails, we need to impose an extra degeneracy condition. For the scalar Cauchy problem \eqref{eq:Cauchy_PB_w} this is normally done through replacing the structural condition \eqref{eq:structural_diffusion_hyp} by the \emph{slow diffusion} hypothesis
\begin{equation}
s>0:\qquad 1+a \leq \frac{s\Phi'(s)}{\Phi(s)} \leq \frac{1}{a}.
\label{eq:slow_diffusion_Sa}
\tag{$S_a$} 
\end{equation}
One readily sees that condition \eqref{eq:slow_diffusion_Sa} implies $\Phi'(0)=0$ and the algebraic behaviour
$ 0 \leq \Phi(1)s^{\frac{1}{a}} \leq \Phi(s) \leq \Phi(1)s^{1+a}$ for small $s$, which usually provides information on the speed of propagation in terms of $a$. However,  since our analysis should include  the Freundlich isotherm $\mathbf{b}_f(\mathbf{z})=(\phi+(1-\phi)|\mathbf{z}|_1^{p-1})\mathbf{z}$ for which the corresponding $\Phi_f(s)$ behaves linearly at infinity,  we cannot assume \eqref{eq:slow_diffusion_Sa} globally in $s>0$ as the lower bound $1<cst\leq \frac{s\Phi'(s)}{\Phi(s)}$ is not admissible for large $s$. In fact, given that the degeneracy is essentially a local feature at the level sets $\{s\approx 0\}$ we could require  condition \eqref{eq:slow_diffusion_Sa}  to hold only  for $0<s\leq s_0$  which is certainly true for our Freundlich isotherm with $1+a=1/p>1$. However, we   prefer to avoid this technical path and, instead,  impose the less restrictive degeneracy condition
\begin{equation}
\forall s>0:\qquad \int_0^s\Phi(s')ds'\geq c\,  \Phi(s)^{\frac{m+1}{m}} \qquad\mbox{for some }m>1.
\tag{$S_m$}
\label{eq:slow_diffusion_Phi_m}
\end{equation}
This implies that  $\Phi'(0)=0$ and is valid for the pure PME nonlinearity $\Phi(s)=s^m$ with $c=m+1$ if $m>1$. 

For technical reasons it will be convenient to reformulate 
\eqref{eq:slow_diffusion_Phi_m} in terms of the original concentration $\mathbf{z}=\mathbf{b}^{-1}(\mathbf{u})$. It is easy to see that the change of variables $r=\Phi(s)=\beta^{-1}(s)$ turns   \eqref{eq:slow_diffusion_Phi_m} into the equivalent condition
\begin{equation}
\forall r>0:\qquad f(r):=r\beta(r)-\int_0^r\beta(r')dr'\geq cr^{\frac{m+1}{m}}\qquad\mbox{for some }m>1,
\tag{$S'_m$}
\label{eq:slow_diffusion_beta_m}
\end{equation}
from which it follows that $\beta'(0)=\infty$, as expected since $\beta=\Phi^{-1}$. An explicit computation shows that \eqref{eq:slow_diffusion_beta_m} holds true globally in $r>0$ for the Freundlich isotherm $\beta_f(r)=\phi r+(1-\phi)r^p$ with $c=(1-\phi)^{-1/p}$ and $m=\frac{1}{p}>1$, showing that \eqref{eq:slow_diffusion_Phi_m}, equivalently \eqref{eq:slow_diffusion_beta_m}, is indeed weaker than \eqref{eq:slow_diffusion_Sa}. %This precise exponent $m=1/p$ was to be expected, since for small $s\searrow 0\Leftrightarrow r\searrow 0$ we have $\beta_f(s)\sim (1-\phi)s^p$, and $\Phi_f(s)=\beta_f^{-1}(s)\sim (r/(1-\phi))^{1/p}$.
\par
In the case of pure PME nonlinearity $\Phi(s)=s^m$ the Cauchy problem \eqref{eq:Cauchy_PB_w} has been widely studied and the qualitative and quantitative theory of free boundaries is now well understood, see e.g. \cite{CVW87,CW90} and references therein. Partial results \cite{dPV91,DR03} also hold for general nonlinearities $\Phi(s)$ but to the best of our knowledge one always assumes the degeneracy condition in the form \eqref{eq:slow_diffusion_Sa}, which fails for the Freundlich isotherm. 

We will start our analysis with a standard statement and prove it assuming only the weaker condition \eqref{eq:slow_diffusion_Phi_m}.
\begin{prop}
Assume that conditions \eqref{eq:monotonicity_b}-\eqref{eq:structural_diffusion_hyp} and \eqref{eq:slow_diffusion_Phi_m} hold for some $a\in (0,1)$ and $m>1$, and let the initial datum $0\leq w^0(x)\leq M$ be compactly supported in $ B_{R_0}$ for some $R_0>0$. Then the Cauchy problem  \eqref{eq:Cauchy_PB_w} admits a unique weak energy solution $0\leq w(x,t)\leq M$ and $\|\nabla \Phi(w)\|_{L^2(Q)}\leq C(a,R_0,M,n)$. Moreover,  $w(\,.\,,t)$ is compactly supported for all $t>0$, the free boundary $\Gamma(t)=\partial \operatorname{supp}w(\,.\,,t)$ propagates with finite speed, and
\begin{equation}
\forall t\geq 0:\qquad \operatorname{supp}\, w(\,.\,,t)\subseteq B_{R(t)}\quad \mbox{with }R(t):=R_0+C_1t^\lambda \, ,
\label{eq:estimate_FB_propagation}
\end{equation}
 with some constants $C_1(m,a,M,R_0)>0$ and $\lambda=\lambda(m,n)>0$.
\label{prop:FB_Cauchy_PB_w_Rn}
\end{prop}
\begin{proof}
Existence and uniqueness are proven in \cite{Va07}. Since $0\leq w^0\leq M$, the comparison principle gives  $0\leq w\leq M$ in $Q$. The $L^2(Q)$-bound for $\nabla\Phi(u)$ easily follows from letting $t\to\infty$ in the classical energy identity
\begin{equation}
\int_{B_R}\Psi(w(t,x))\,\mathrm{d}x+\int\limits_0^t\int\limits_{B_R}|\nabla \Phi(w(x,\tau))|^2\,\mathrm{d}x\,\mathrm{d}\tau=\int_{B_R}\Psi(w^0(x))\,\mathrm{d}x\,
\label{eq:energy_identity_w}
\end{equation}
where $\Psi(s):=\int_0^s\Phi(s')\,\mathrm{d}s'$ (see \cite{Va07} for details). Indeed with our assumptions $w^0$ is bounded and compactly supported hence $\|\Psi(w^0)\|_{L^1(\R^n)}\leq \Psi(M)\operatorname{meas}(B_{R_0})= C(a,M,R_0,n)$.\par
As for the speed of propagation of the support, we go back to the original concentration formulation and  recall from \cite{DV85}  two results  based on the energy methods introduced in \cite{A81}.  To verify that the assumptions in \cite{DV85} are satisfied,
we set $z:= \Phi(w)=\beta^{-1}(w)$ and observe that $\partial_t\beta(z)=\Delta z$ in $Q$. Making the change of variables $r=\Phi(s)$ in  \eqref{eq:slow_diffusion_beta_m} shows that $f(z(x,t))=\Psi(w(x,t))$, and the energy identity \eqref{eq:energy_identity_w} readily gives
$$
E(z;t):=\sup\limits_{0\leq \tau\leq t} \int_{\R^n}f(z(x,\tau))\,\mathrm{d}x + \int\limits_0^t\int\limits_{\R^n}|\nabla z(x,\tau)|^2\,\mathrm{d}x\,\mathrm{d}\tau \leq C(a,M,R_0,n).
$$
Defining $j(r):=\int_0^r\beta(r')\,\mathrm{d}r'$ and making the change $r=\Phi(s)$, we  obtain for all $0\leq w_1,w_2\leq M$ the bound
\begin{align*}
\left|j(\Phi(w_1))-j(\Phi(w_2))\right|	 =\left|\int_{\Phi(w_2)}^{\Phi(w_1)} \beta(r)\,\mathrm{d}r\right|	&=\left|\int_{w_1}^{w_2}s\Phi'(s)\,\mathrm{d}s\right|\\
  & \leq \left|\int_{w_1}^{w_2}\frac{1}{a}\Phi(s)\,\mathrm{d}s\right| \leq \frac{\Phi(M)}{a}|w_1-w_2|.
\end{align*}
Since $0\leq w(x,t)\leq M$ and $z=\Phi(w)\leq \Phi(M)\leq C(a,M)$ we have in particular
$$
\forall t_1,t_2\in [0,T]:\qquad \|j(z(t_1))-j(z(t_2))\|_{L^1(\R^n)}
  \leq  C\|w(t_1)-w(t_2)\|_{L^1(\R^n)}.
$$
Given our assumptions on $w^0$, we can infer from the classical theory \cite{Va07} for the scalar Cauchy problem \eqref{eq:Cauchy_PB_w}  that  $w\in \mathcal{C}([0,T];L^1(\R^n))$. This  in turn  yields $j(z)\in \mathcal{C}(0,T;L^1(\R^n))$.\par

The energy estimate $E(z;t)\leq C$ and the continuity of $j(z)$ allow us to apply \cite[Corollary 3.1]{DV85} and we conclude that $z(\,.\,,t)$ is compactly supported, satisfies \eqref{eq:estimate_FB_propagation}. Similarly, applying \cite[Theorem 3.1]{DV85} shows that the free boundary $\Gamma(t)=\partial \operatorname{supp}w(\,.\,,t)$ propagates with finite speed and the proof is complete.
\end{proof}
We can now establish the corresponding result on the multicomponent Cauchy problem \eqref{eq:Cauchy_PB_u}.
\begin{theorem}[Free Boundary solutions]
Let conditions \eqref{eq:monotonicity_b}-\eqref{eq:structural_diffusion_hyp}-\eqref{eq:convexity_hyp} and \eqref{eq:slow_diffusion_beta_m} hold for some $a\in(0,1)$ and $m>1$. Assume  that $\mathbf{u}^0\in L^\infty(\mathbb{R}^n)$ is componentwise non-negative with $w^0=|\mathbf{u}^0|_1\leq M$, and such that $\operatorname{supp}w^0\subseteq B_{R_0}$ for some $R_0>0$. Then there exists a unique non-negative very weak solution $\mathbf{u}\in L^{\infty}(Q)$ to \eqref{eq:Cauchy_PB_u}. Moreover,
\begin{enumerate}
 \item[(i)]
 $w=|\mathbf{u}|_1$ is the unique weak energy solution to \eqref{eq:Cauchy_PB_w}, $0\leq w\leq M$, and
 $$
 \|\nabla\Phi(w)\|_{L^2(Q)}\leq C(a,M,R_0)
 $$
 \item[(ii)]
 $\mathbf{u}$ is a local energy solution to \eqref{eq:Cauchy_PB_u} in the sense that for all $T>0$ we have
 $$
 \|\nabla(\ro u_i)\|_{L^2(Q^T)}\leq C(a,M,R_0,T)\qquad \forall i=1,\ldots N \, , 
 $$
 where $\ro=\frac{\Phi(w)}{w}$.
 \item[(iii)]
 $\operatorname{supp}\,w(\,.\,,t)$ propagates with finite speed, and
 $$
 \forall\, t\geq 0,\,i=1\ldots N: \qquad \operatorname{supp}u_i(\,.\,,t)\subseteq \operatorname{supp} w(\,.\,,t)\subseteq B_{R(t)} 
 $$
 with $R(t)=R_0 + C_1t^\lambda$ for some $C_1(m,a,M,R_0)>0$ and $\lambda=\lambda(m,n)>0$ only.
 \item[(iv)]
 There is $\alpha=\alpha(a,n)\in(0,1)$ such that $\mathbf{u}$ is $(\alpha,\alpha/2)$-H\"older continuous in any strip $Q^{\tau,T}=\R^n\times(\tau,T)$, $0<\tau<T$, and
%  $$
%  \|\mathbf{u}\|_{\mathcal{C}^{\alpha,\alpha/2}(Q^{\tau,T})}\leq C(M,N,\tau,T)
%  $$ 
$$
\|\mathbf{u}\|_{C^{\alpha,\alpha/2}(Q^{\tau,T})}\leq C(1+1/\sqrt{\tau}).
$$
for some $C(a,T,n,N,M)>0$ only.
 \item[(v)]
 If $\Phi$ is smooth in $\R^+$ then $u_i$ is smooth in $\{u_i>0\}\cap\{t>0\}$.
\end{enumerate}
\label{theo:free_boundaries}
\end{theorem}
\begin{remark} As in Theorem~\ref{theo:exist_sols_Dirichlet_Cauchy_PB}, the structural condition \eqref{eq:convexity_hyp}  is only needed to get (ii) and can be relaxed by restricting \eqref{eq:Cauchy_PB_u} to very weak solutions instead of  energy solutions. Moreover, as in Proposition~\ref{prop:boundary_regularity}, the H\"older regularity estimate (iv) can be extended up to $t=0^+$ if we  assume further $\mathcal{C}^{\beta}(\R^n)$ regularity from $\mathbf{u}^0$. Note in particular  that in (iii) we only claim that  $w$ has finite speed of propagation but not that the individual species propagate with finite speed. In fact the support of each species should in general be discontinuous in time, see the discussion at the end of this section.
\end{remark}
% . Consider for example the simple case of two species $\mathbf{u}=(u_1,u_2)$ only, with $\operatorname{supp} u_1^0=B_{R_0}$ and $\operatorname{supp} u_2^0=B_{3R_0}$. Then in $B_{3R_0}$ we have $w^0=u_1^0+u_2^0>0$, thus $\ro^0=\Phi(w^0)/w^0>0$ there. Assuming a minimum of regularity in time (for example in the form of Proposition~\ref{prop:boundary_regularity}
% ) we should have $\ro(\
% ,.\,,t)>0$ at least in $B_{2R_0}\times(0,T)$ for small $T>0$. But now $u_1$ solves the equation $\partial_t u_1=\Delta(\ro u_1)=\dive(\ro u_1)+\dive(u_1\nabla\ro)$, which is uniformly parabolic in $B_{2R_0}\times(0,T)$. As a consequence the initial support $B_{R_0}$ of $u_1$ jumps instantaneously at $t=0^+$ to $u_1(\,.\,,t)>0$ in (at least) $B_{2R_0}$ for small $t>0$. This instantaneous invasion phenomenon is beyond the scope of this paper and will be investigated in an upcoming work
%
\begin{proof}
Arguing exactly as in the proof of Proposition~\ref{prop:uniqueness_weak_sols} but choosing now test function $ \theta\in \mathcal{C}^{\infty}_c(\R^n\times(0,\infty))$ it is easy to show uniqueness within the class of very weak solutions. It is therefore enough to prove existence in finite time intervals $[0,T]$ for any fixed $T>0$. 

Given that the free boundary should a priori propagate with finite speed, solutions to the Cauchy problem in the whole space should agree with solutions of the Cauchy-Dirichlet problem in $B_R\times(0, T)$ with zero boundary conditions, as long as $R>0$ is large enough so that the free boundary stays at a positive distance from $\partial B_R$ for all $t\leq T$. We should therefore be able to construct solutions to the Cauchy problem in $\R^n\times(0, T)$ by considering auxiliary Dirichlet problems  in large balls and  using Theorem~\ref{theo:exist_sols_Dirichlet_Cauchy_PB}.\\
From Proposition~\ref{prop:FB_Cauchy_PB_w_Rn} we can define the unique solution  $\overline{w}$ of the Cauchy problem \eqref{eq:Cauchy_PB_w} in $\R^n\times(0,\infty)$ with  initial data $w^0:=|\mathbf{u}^0|_1$. For any fixed $T>0$ choose $R>0$ large enough so that $R(T)\leq R/2$ and
$$
\forall t\leq T:\qquad \operatorname{supp}\,\overline{w}(\,.\,,t)\subseteq B_{R(T)}\subseteq B_{R/2},
$$
where $R(T)=R_0+C_1T^\lambda$ and the constants $C_1$ and $\lambda$ are as in Proposition~\ref{prop:FB_Cauchy_PB_w_Rn}. Next, let $\mathbf{u}=\mathbf{u}_R$ be the unique solution  to problem \eqref{eq:PB_u}  in $B_R\times(0,T)$ corresponding to the  initial data $\mathbf{u}^0$ and zero boundary values on $\partial B_R$ and  given by Theorem~\ref{theo:exist_sols_Dirichlet_Cauchy_PB}. It also follows from Theorem~\ref{theo:exist_sols_Dirichlet_Cauchy_PB}  that  $w:=|\mathbf{u}|_1$ is a weak solution to the corresponding Cauchy-Dirichlet problem in $B_R\times(0,T)$ with zero boundary values. Given the definition of $R$ we thus see that $\overline{w}$ remains at a distance $R/2$ away from $\partial B_R$. It is easy to see that the restriction $\overline{w}_{|B_R\times(0,T)}$ is also a weak solution to the same Cauchy-Dirichlet problem as $w$. By standard uniqueness theorem for weak solutions of \eqref{eq:PB_w} we conclude that $w=\overline{w}$ in $B_R\times(0,T)$. In particular
$$
 \forall\, t\leq T,\,i=1\ldots N: \qquad \operatorname{supp}\,u_i(\,.\,,t)\subseteq\operatorname{supp}\,w(\,.\,,t) = \operatorname{supp}\overline{w}(\,.\,,t)\subseteq B_{R(t)}
$$
where $R(t)=R_0+C_1t^\lambda$ and the distance between the support of $\mathbf{u}$ and $\partial B_R$ is at least $R/2>0$ for all $t\leq T$. Extending $\mathbf{u}$ and $w$  outside $B_{R_0}$ by zero for all $t\in(0,T)$ it is then a simple exercise to verify that these extensions satisfy the weak formulations of \eqref{eq:Cauchy_PB_u} and \eqref{eq:Cauchy_PB_w} in the whole space, whence existence of free-boundary solutions in $\R^n\times(0,T)$ for arbitrary $T>0$. The energy estimate (i) and propagation properties (iii) immediately follow from the definition of $\overline{w}$ and Proposition~\ref{prop:FB_Cauchy_PB_w_Rn}.\\
 For fixed $T>0$ take now $R>0$ large enough so that $\mathbf{u}$ stays supported in $B_R$ for all $t\leq T$. Viewing $\mathbf{u}$ as the unique solution to the Cauchy-Dirichlet problem in $B_{R+1}$ and taking $\Omega=B_{R+1}$, $\Omega'=B_R$ in Theorem~\ref{theo:exist_sols_Dirichlet_Cauchy_PB} we have $d'=1$ in \eqref{eq:energy_very_weak_ui}, thus
$$
\|\nabla(\ro u_i)\|_{L^2(\R^n\times(0,T))}=\|\nabla(\ro u_i)\|_{L^2(B_R\times(0,T))}\leq C(a,n,N,M,T)
$$
 as claimed in (ii).

Assertion (iv) is proven similarly by  considering the Cauchy-Dirichlet problem in $B_{R+1}$, choosing $\Omega^\prime=B_{R}\subset B_{R+1}=\Omega$ in Theorem~\ref{theo:exist_sols_Dirichlet_Cauchy_PB}  and taking $d'=1$ in estimate \eqref{eq:uniform_Holder}.

To prove (v) we use a  local bootstrap argument. If $w>0$ in some $B_r(x_0)\times(t_0-\tau,t_0+\tau)$, $t_0>0$, then in particular the pressure $p=\Phi'(w)>0$ there. Since $w$ is H\"older continuous and $\Phi$ is smooth also $p$ is  H\"older continuous. Moreover, $w$ solves a uniformly parabolic equation in divergence form: $\partial_t w=\Delta\Phi(w)=\dive(p\nabla w)$. Hence $w\in C^{1+\beta}$ for some $\beta$. 
%\textcolor{red}{Is this true? certainly, but I'd better check with Tuomo}.
By bootstrapping we immediately see that $w$ is locally smooth. Consider now any species $u_i$ and observe that if $u_i>0$ in $B_r(x_0)\times(t_0-\tau,t_0+\tau)$ then $w=|\mathbf{u}|_1\geq u_i>0$ and also $\ro=\Phi(w)/w>0$. Since $u_i$ solves $\partial_t u_i=\Delta(\ro u_i)=\dive(\ro\nabla u_i)+\dive(u_i\nabla\ro)$ with now smooth coefficients we conclude that $u_i$ is smooth and the proof is complete.
\end{proof}

Although degenerate, problem \eqref{eq:Cauchy_PB_w} is nonetheless diffusive in nature and we expect that the information cannot propagate backwards as confirmed by the following result. 
\begin{prop}[Persistence property]
Under the hypotheses of Theorem~\ref{theo:free_boundaries}, let us further assume that for any $M >0$ there exists $\overline{a}=\overline{a}(M )>0$ such that
$$
\forall s\in(0,M ]:\qquad 1+\overline{a}\leq \frac{s\Phi'(s)}{\Phi(s)}.
$$
Then the support $\Omega(t):=\{x:\,w(x,t)>0\}$ of $w=|\mathbf{u}|_1$ is non-contracting in time.
\end{prop}

\begin{remark}
This (weaker) degeneracy condition is easily checked for the Freundlich isotherm. In fact,  $s\Phi'(s)/\Phi(s)=\beta(r)/(r\beta'(r))$ is strictly greater than one for $\beta_f(r)=\phi r+(1-\phi)r^p$ in any finite interval (but not in the limit $s\to\infty$ because $\Phi_f(s)$ becomes linear).
\end{remark}
\begin{proof}
Dahlberg and Kenig \cite{DK86}  proved that nonnegative solutions to \eqref{eq:Cauchy_PB_w} satisfy certain Harnack inequality provided the strong condition \eqref{eq:slow_diffusion_Sa} holds. In this case positivity of $w$ at $(x_0,t_0)$ implies positivity at $(x_0,t)$ for all later $t\geq t_0$ and the support is  non-contracting. Unfortunately, as already discussed, \eqref{eq:slow_diffusion_Sa} does not hold globally in $s>0$.\\
In order to tackle this technical detail we recall from Theorem~\ref{theo:free_boundaries} that $0\leq w\leq M$, with $M=\|w^0\|_{L^{\infty}(\R^n)}$. 
Now the assumption on $\Phi(s)$ allows one to construct a $\mathcal{C}^1$-function $\overline{\Phi}$  which satisfies \eqref{eq:slow_diffusion_Sa} globally in $s>0$ for some $a=a(M)\in(0,1)$ and  such that $\overline{\Phi}(s)=\Phi(s)$ for all $s\in [0,M]$. By construction $w$ is a solution to $\partial_t w=\Delta \overline{\Phi}(w)$, and the assertion follows from \cite{DK86} (for a precise statement see e.g. \cite[Corollary 1.5]{dPV91}).
\end{proof}

We end this section with a ``divide and rule'' result.
\begin{prop}
Assume that \eqref{eq:monotonicity_b}, \eqref{eq:structural_diffusion_hyp}, \eqref{eq:convexity_hyp} and \eqref{eq:slow_diffusion_beta_m} hold true for some $a\in(0,1)$ and $m>1$. Let $k\in \{1,\ldots,N\}$, $\hat{\mathbf{u}}^0=(u_1^0,\ldots,u_k^0)\in L^{\infty}(\R^n;\R^k)$ and $\check{\mathbf{u}}^0=(u_{k+1}^0,\ldots,u_N^0)\in L^{\infty}(\R^n;\R^{N-k})$ be non-negative and compactly supported, with $d=\dist(\operatorname{supp}(\hat{\mathbf{u}}^0),\operatorname{supp}(\check{\mathbf{u}}^0))>0$. Set $\mathbf{u}^0=(\hat{\mathbf{u}}^0,\check{\mathbf{u}}^0)\in L^{\infty}(\R^n;\R^{k+(N-k)})$. Moreover, let $\hat{\mathbf{u}}(x,t)$ be the unique solution of the $k$-dimensional Cauchy problem \eqref{eq:Cauchy_PB_u} with the initial data $\hat{\mathbf{u}}^0$, $\check{\mathbf{u}}(x,t)$ the unique solution of the $(N-k)$-dimensional Cauchy problem \eqref{eq:Cauchy_PB_u} with the initial data $\check{\mathbf{u}}^0$ and assume finally  that $\mathbf{u}(x,t)$ is the unique solution of the $N$-dimensional Cauchy 
problem \eqref{eq:Cauchy_PB_u} with the
initial data $\mathbf{u}^0=(\hat{\mathbf{u}}^0,\check{\mathbf{u}}^0)$. Then there exists $T>0$ such that $\mathbf{u}\equiv (\hat{\mathbf{u}},\check{\mathbf{u}})$ in $Q^T=\R^n\times(0,T)$. More precisely, if $\hat{w}$ and $\check{w}$ are the two solutions of the scalar Cauchy problem \eqref{eq:Cauchy_PB_w} with the respective initial data $\hat{w}^0=|\hat{\mathbf{u}}^0|_{l^1(\R^k)}$ and $\check{w}^0=|\check{\mathbf{u}}^0|_{l^1(\R^{N-k})}$, then
$$
T=\inf\{t\geq 0:\   \operatorname{supp} \hat{w}(\,.\,,t)\cap \operatorname{supp} \check{w}(\,.\,,t) \neq \emptyset\}\in(0,\infty].
$$
\label{prop:divide_rule}
\end{prop}
\begin{remark}
As is clear from the above definition, $T$ is the first time when the supports of $\hat{w},\check{w}$ touch. Our statement can be reformulated simply  as: if the initial data can be separated into two distinct patches of $k$ and $N-k$ species then it suffices to solve  two independent  lower dimensional systems of order $k$ and $N-k$ as long as their respective supports do not touch. Note that we do not claim anything for what happens after $t=T$ because when the supports touch the two patches start interacting and the situation becomes more involved.
\end{remark}
\begin{proof}
From Theorem~\ref{theo:free_boundaries} it is clear that the supports of $\hat{\mathbf{u}}(\,.\,,t),\check{\mathbf{u}}(\,.\,,t)$ propagate with finite speed, which implies  that $T>0$. Letting now $\tilde{\mathbf{u}}=(\hat{\mathbf{u}},\check{\mathbf{u}})\in L^{\infty}(\R^n\times(0,T);\R^{k+(N-k)})$, we observe that by definition of $T$
$$
(x,t)\in \operatorname{supp}(\hat{\mathbf{u}})\cap\{t\leq T\} \quad \Rightarrow \quad 
\left\{
\begin{array}{ll}
\tilde{\mathbf{u}}=(\hat{\mathbf{u}},0)\\
|\tilde{\mathbf{u}}|_{l^1(\R^N)}=|(\hat{\mathbf{u}},0)|_{l^1(\R^{k+(N-k)})}=\left|\hat{\mathbf{u}}\right|_{l^1(\R^{k})}
\end{array}
\right.
$$
and
$$
(x,t)\in \operatorname{supp}(\check{\mathbf{u}})\cap\{t\leq T\} \quad \Rightarrow \quad 
\left\{
\begin{array}{ll}
\tilde{\mathbf{u}}=(0,\check{\mathbf{u}})\\
|\tilde{\mathbf{u}}|_{l^1(\R^N)}=|(0,\check{\mathbf{u}})|_{l^1(\R^{k+(N-k)})}=\left|\check{\mathbf{u}}\right|_{l^1(\R^{N-k})}
\end{array}
\right. .
$$
It is then easy to check that $\tilde{\mathbf{u}}=(\hat{\mathbf{u}},\check{\mathbf{u}})$ solves the global $N$-dimensional system in $Q^T$ (but a priori \emph{not} for later times). From uniqueness in Theorem~\ref{theo:exist_sols_Dirichlet_Cauchy_PB} (restricted to finite time intervals) we conclude that $\tilde{\mathbf{u}}=\mathbf{u}$ in $Q^T$.
\end{proof}
%
% In the previous Proposition~\ref{prop:divide_rule} it might actually happen that the supports of the two patches never touch, i-e $T=\infty$. Bounding $T<\infty$ corresponds to a lower estimate on the speed of propagation, thus ensuring that the supports eventually touch in finite time. In the case of the pure PME nonlinearity $\Phi(s)=s^m$ $(m>1)$ it is for example well known that the support is non-contracting in time, any arbitrary point $x_0\notin \operatorname{supp}\,w(\,.\,,t_0)$ is reached in finite time, and thus $T<\infty$ here. To the best of our knowledge no such result is available for general nonlinearities $\Phi(s)$ unless \eqref{eq:slow_diffusion_Sa} is assumed, see again \cite{dPV91}.\par 
%

One can refine Proposition~\ref{prop:divide_rule}  by considering an arbitrary number $j$, say, of initial patches as 
follows: 1) as long as the supports do not intersect, solve $j$ independent systems, 2) when two or more patches touch, glue them together into a single  higher-dimensional patch and resume with $j'<j$  patches, 3) keep iterating as long as there is more than one patch left. 

% Using Proposition~\ref{prop:divide_rule} as a ``base brick'' one could make this rigorous by induction on $j$, but due to the lack of space we shall refrain from taking this technical path.

This ``divide and rule'' behaviour should lead to infinite speed of propagation (discontinuity in time) of the supports of each individual species, even though the support of $w=|\mathbf{u}|_1$ does propagate with finite speed as stated in Theorem~\ref{theo:free_boundaries}. Consider for example the case of only two species $\mathbf{u}=(u_1,u_2)$  with the initial (compact) supports at a positive distance from each other. We know from Proposition~\ref{prop:divide_rule} that we only have to solve two scalar problems as long as the supports do not intersect. Assume that this happens at $t=T$ and that at this particular moment the two supports look like two tangent balls (thus only one species is present in each ball). If at least one of the balls were expanding at time $t=T^-$,  then for $t=T^+$ the balls should intersect and the support of $w$ would thus look like an $8$-shaped domain with a thin but non-empty interior neck connecting 
the balls. Moreover, the diffusion coefficient 
$\ro =\frac{\Phi(w)}{w}$ becomes positive in the entire $8$-shaped domain, in particular in the neck. Since $\partial_t u_i=\dive(\ro \nabla u_i)+(\ldots)$ with the support of $u_i$ in either of the two balls at time $t=T$, the diffusion occurring in the neck  will ensure infinite speed of propagation between the two balls. Thus the support of $u_i$ should jump from only one ball at $t=T$ to the whole $8$-shaped domain at $t=T^+$. This instantaneous invasion phenomenon is beyond the scope of this paper and will be investigated elsewhere.

\bibliographystyle{alpha}

\end{document}